\documentclass[12pt]{article}

\usepackage{graphicx}
\usepackage{geometry}
\usepackage{amsfonts}
\usepackage{amsmath}
\usepackage{enumitem}
\usepackage{amsthm}
\usepackage{hyperref}

\DeclareMathOperator{\supp}{supp}
\DeclareMathOperator{\gain}{gain}
\DeclareMathOperator{\bonus}{bonus}

\newenvironment{myfigure}[1][1]{\bigskip}{\bigskip}

\begin{document}

\setcounter{page}{1} \setcounter{section}{0}
\newtheorem{theorem}{Theorem}[section]
\newtheorem{lemma}[theorem]{Lemma}
\newtheorem{corollary}[theorem]{Corollary}
\newtheorem{proposition}[theorem]{Proposition}
\newtheorem{observation}[theorem]{Observation}
\newtheorem{definition}[theorem]{Definition}
\newtheorem{claim}{Claim}
\newtheorem{conjecture}[theorem]{Conjecture}
\newtheorem{problem}[theorem]{Problem}

\title{3-Flows with Large Support}

\author{
	Matt DeVos\thanks{Department of Mathematics, Simon Fraser University, Burnaby, B.C., Canada V5A 1S6, mdevos@sfu.ca.}
\and
	Jessica McDonald\thanks{Department of Mathematics and Statistics, Auburn University, Auburn, AL, USA 36849, mcdonald@auburn.edu.}
\and
	Irene Pivotto\thanks{School of Mathematics and Statistics, University of Western Australia, Perth, WA, Australia 6009, irene.pivotto@uwa.edu.au.}
\and
	Edita Rollov\'a\thanks{European Centre of Excellence, NTIS- New Technologies for Information Society, Faculty of Applied Sciences, University of West Bohemia, Pilsen, edita.rollova@gmail.com.}
\and
	Robert \v{S}\'amal \thanks{Computer Science Institute of Charles University, Prague, samal@iuuk.mff.cuni.cz.}
}

\date{}

\maketitle

\begin{abstract}
We prove that every 3-edge-connected graph $G$ has a 3-flow $\phi$ with the property that $| \supp( \phi) | \ge \frac{5}{6} |E(G)|$.  The graph $K_4$
demonstrates that this $\frac{5}{6}$ ratio is best possible; there is an infinite family where $\frac 56$ is tight.
\end{abstract}

\newpage
\tableofcontents
\newpage

\section{Introduction}

Throughout the paper we permit graphs to have both multiple edges and loops.  If $G$ is an oriented graph and $\Gamma$ is an additive abelian group, then we define a function $\phi : E(G) \rightarrow \Gamma$ to be a \emph{flow} if it satisfies the following rule at every vertex $v \in V(G)$:
\[ \sum_{e \in \delta^+(v)} \phi(e) - \sum_{e \in \delta^-(v)} \phi(e)  = 0,\]
where $\delta^+(v)$  ($\delta^-(v)$) denotes the set of edges directed away from (toward) the vertex $v$.
The flow $\phi$ is \emph{nowhere-zero} if $0 \not\in \phi(E(G))$ and it is called a $k$-\emph{flow} for a positive integer $k$ if $\Gamma = \mathbb{Z}$ and $| \phi(e) | \le k-1$ for every $e \in E(G)$. The \emph{support} of $\phi$ is the set of all edges of $G$ with $\phi(e)\neq 0$, and is denoted by $\supp (\phi)$. In~\cite{Tutte54} Tutte initiated the study of nowhere-zero flows by proving the following duality theorem.

\begin{theorem}[Tutte]\label{th:Tutte-col}
\label{duality}
If $G$ and $G^*$ are dual planar graphs, then $G^*$ has a proper $k$-colouring if and only if $G$ has a nowhere-zero $k$-flow.
\end{theorem}

Based in part on this duality, Tutte (\cite{Tutte54, Tutte66}) made three lovely conjectures concerning the existence of nowhere-zero flows.  These conjectures, known as the 5-Flow, 4-Flow, and 3-Flow conjectures have motivated a great deal of research on this subject, but despite this all three remain unsolved.

Our approach here will be to relax the notion of nowhere-zero and instead look for flows which have large support. Our main theorem is the following bound for 3-flows in 3-edge-connected graphs.

\begin{theorem}
\label{main1}
Every 3-edge-connected graph $G$ has a 3-flow $\phi$ satisfying
\[ |\supp (\phi) | \ge \tfrac{5}{6} |E(G)|.\]
\end{theorem}

Since the graph $K_4$ does not have a nowhere-zero 3-flow, but $K_4 - e$ does (for any edge $e$ of $K_4$), the ratio $\frac{5}{6}$ in this
theorem is best possible. The family of tight examples is much larger, though. Let us call \emph{tripod} a graph obtained from $K_3$ by adding a pendant
edge to every vertex. So tripod has three leaves (i.e., vertices of degree~1); identifying the leaves of a tripod produces a~$K_4$. Suppose $G$ is a  graph obtained from a set of tripods by identifying their leaves (in any desired way).
Consider any 3-flow of $G$; restricting such a flow to a single tripod and then contracting all edges outside the tripod produces a 3-flow of $K_4$. Therefore $G$ has no flow with support larger than $\tfrac{5}{6} |E(G)|$.
An easy way to get
large 3-edge-connected graphs as a union of tripods is to start with a 3-edge-connected bipartite graph $H=(U,V,E)$ such that all vertices in~$U$ have degree~3. Then we truncate each
vertex~$v$ in~$U$, which turns it into a triangle. This triangle together with the original neighbors of~$v$ form a tripod.

When we restrict our attention to planar graphs, our theorem has the following corollary (proved in the following section).

\begin{corollary}
\label{planarcor}
If $G$ is a simple planar graph, then there exists a function $f : V(G) \rightarrow \{1,2,3\}$ so that the number of edges $uv$ with $f(u) = f(v)$ is at most $\frac{1}{6} |E(G)|$.
\end{corollary}

Note that the graph $K_4$ also demonstrates that the above corollary gives a best possible bound.  In fact, this corollary is not a new result -- it is also a corollary to the Four Colour Theorem.   To see this, let $f : V(G) \rightarrow \{1,2,3,4\}$ be a proper 4-colouring of $G$, and assume (without loss) that the number of edges with one end of colour 3 and one of colour 4 is at most $\frac{1}{6}|E(G)|$.  Now the function $g : V(G) \rightarrow \{1,2,3\}$ given by $g(v) = \min \{ f(v), 3 \}$ is a 3-colouring with the desired properties.  To our knowledge, our argument gives the first proof of this result not relying on the Four Colour Theorem.

Although the $\frac{5}{6}$ ratio in Theorem~\ref{main1} is best possible, it seems quite possible that the same bound holds more generally for graphs which are 2-edge-connected.  Unlike most results in the realm of nowhere-zero flows, our theorem on 3-edge-connected graphs does not obviously give a similar result for 2-edge-connected graphs.  The best bound we have for 3-flows in 2-edge-connected graphs is the following result due to Tarsi \cite{tarsi} (proved in Section~\ref{sec:bounds}). An earlier version of this paper instead included a result by Kr\'al'~\cite{Kral} with $\tfrac{3}{4}$ in place of $\tfrac{4}{5}$.

\begin{theorem}[Tarsi]
\label{2ec3flow}
Every 2-edge-connected graph $G$ has a 3-flow $\phi$ with
\[ | \supp(\phi) | \ge \tfrac{4}{5} |E(G)|. \]
\end{theorem}

More generally, we are interested in finding bounds on the maximum support of a $k$-flow in a $t$-edge-connected graph.  For a graph $G$ and a positive integer $k$, let $h(G,k)$ be the maximum of $\frac{ | \supp(\phi) |}{ |E(G)| }$ over all possible $k$-flows $\phi$.  Then, for a positive integer $t$, we define $h(t,k)$ to be the infimum of $h(G,k)$ over all $t$-edge-connected graphs.  So $h(t,k)$ is the best lower bound for the maximum ratio of edges covered by a $k$-flow over all $t$-edge-connected graphs.  It is immediate that $h(1,k) = 0$ and $h(t,1) = 0$ for all $k,t$.  The following table indicates our present state of knowledge of $h(t,k)$ for $2 \le t,k \le 6$.

\begin{figure}[htp]
\centerline{\includegraphics[width=12cm]{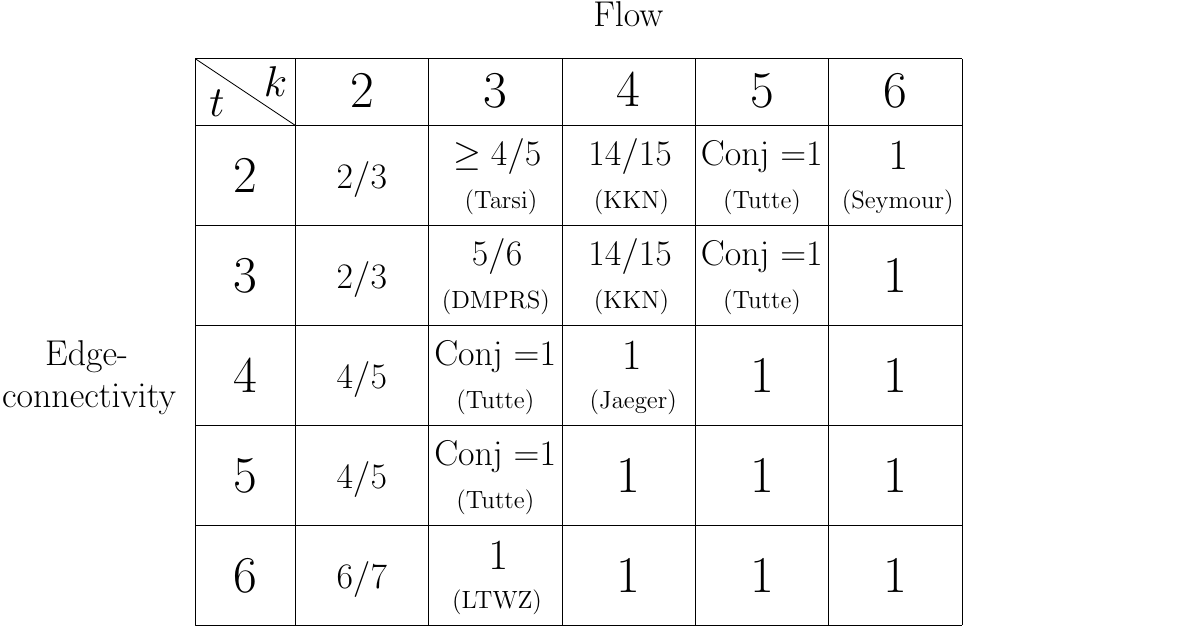}}
\caption{The function $h(t,k)$}
\label{table}
\end {figure}

Tutte's famous 5-Flow and 3-Flow Conjectures are equivalent to the assertions that $h(2,5) = 1$ and $h(4,3) = 1$ as shown in the table.  Several famous
  theorems on flows appear here as well.  For instance, Seymour's 6-Flow Theorem~\cite{Seymour81} is equivalent to $h(2,6) = 1$, Jaeger's 4-Flow
  Theorem~\cite{Jaeger79} is equivalent to $h(4,4) = 1$, and the recent result of Lov\'asz, Thomassen, Wu, and Zhang on 3-Flows~\cite{LTWZ2013} is
  equivalent to $h(6,3) = 1$.  The indicated $h(2,4) = \frac{14}{15}$ result is a straightforward consequence of a variant of a theorem of Kaiser,
  Kr\'al' and Norine~\cite{KKN2006} as we show in the next section.  This table suggests that $h(2j, 2) = h(2j+1, 2) = \frac{2j}{2j+1}$.  This is indeed
  true, and will also be established in the following section by way of some standard techniques.  Since $h(t,k)$ is nondecreasing in both $t$ and $k$, all values
  of this function are known apart from the pairs $(t,k) = (2,3)$, $(2,5)$, $(3,5)$, $(4,3)$, and $(5,3)$.

In this paper we will prove that $h(3,3)=5/6$. In fact, our main theorem has a somewhat stronger ``choosability'' form related to group-connectivity as introduced by
  Jaeger, Linial, Payan, and Tarsi~\cite{JLPT1992}.  Instead of insisting that the function $\phi$ is a flow, we may instead ask for the sum involved at the vertex~$v$,
  (i.e., $\sum_{e \in \delta^+(v)} \phi(e) - \sum_{e \in \delta^-(v)} \phi(e)$) to take on certain prescribed values at each vertex $v$.  Let us return to a general setting to put these definitions in place.  Assume that $G$ is an oriented graph, let $\Gamma$ be an abelian group (written additively) and let $\phi : E(G) \rightarrow \Gamma$.  The \emph{boundary} of~$\phi$ is the function $\partial \phi : V(G) \rightarrow \Gamma$ given by the following rule for every $v \in V(G)$:
\[
   \partial \phi( v) = \sum_{ e \in \delta^+(v) } \phi(e) - \sum_{ e \in \delta^-(v) } \phi(e).
\]
We recall that $\delta^+(v)$ and $\delta^-(v)$ denote the set of edges directed away from and toward~$v$ respectively.
If we think of $\phi$ as indicating a circulation of fluid, then $\partial \phi (v)$ tells us how much is leaving the network at~$v$.
Note that by definition, the function $\phi$ is a flow if $\partial \phi$ is identically zero.  If we sum the boundary function $\partial \phi$
over all vertices, then whatever value $x$ is assigned to an edge $e$ will get added once and subtracted once, so it has no effect.  This gives the
following useful identity
\[
   \sum_{v \in V(G)} \partial \phi (v)  = 0
\]
which holds for every function $\phi : E(G) \rightarrow \Gamma$.  In general, we say that a function $\mu : V(G) \rightarrow \Gamma$ is \emph{zero-sum} if $\sum_{v \in V(G)} \mu(v) = 0$. The general form of our main theorem may now be stated as follows.

\begin{theorem}
\label{main2}
If $G$ is an oriented 3-edge-connected graph and $\mu : V(G) \rightarrow \mathbb{Z}_3$ is zero-sum, then there exists $\phi : E(G) \rightarrow \mathbb{Z}_3$ so that
\begin{enumerate}
\item $\partial \phi = \mu$, and
\item $| \supp (\phi) | \ge \frac{5}{6} |E(G)|.$
\end{enumerate}
\end{theorem}

To see that this theorem implies Theorem~\ref{main1}, simply apply it with $\mu = 0$ to choose a flow $\phi: E(G) \rightarrow \mathbb{Z}_3$ with $\supp (\phi)\ge \frac{5}{6}|E(G)|$.  Now using Tutte's equivalence between nowhere-zero modular and integer flows~\cite{Tutte56} (applied on the support of $\phi$) we deduce that $G$ has a 3-flow with support of size at least $\frac{5}{6}|E(G)|$ as desired.

Unlike our earlier theorem, in the case of Theorem~\ref{main2} the assumption of 3-edge-connectivity is necessary.  To see this, take an arbitrary oriented 3-edge-connected graph $G$ with $|E(G)| \equiv 0 \pmod 3$ and modify it by subdividing every edge twice (thus forming a directed path of three edges).  Now define $\mu : V(G) \rightarrow \mathbb{Z}_3$ by the following rule:
\[ \mu(v) = \left\{ \begin{array}{cl}
	1	&	\mbox{if $\deg(v) = 2$}	\\
	0	&	\mbox{otherwise.}
	\end{array} \right. \]
For every 3-edge path $P$ with both interior vertices of degree 2, a straightforward check reveals that every function $\phi : E(G) \rightarrow
  \mathbb{Z}_3$ which satisfies $\partial \phi = \mu$ will have the property that $\phi$ assigns all three edges of $P$ distinct values.  Therefore,
  every such function $\phi$ will satisfy $|\supp(\phi)| = \frac{2}{3} |E(G)|$.
  Consequently, Theorem~\ref{2ec3flow} does not extend from $\mathbb{Z}_3$-flows to $\mathbb{Z}_3$-connectivity.

\section{Flow and Colouring Bounds}\label{sec:bounds}

In this section we will prove Corollary~\ref{planarcor} and some of the results stated in the table of Figure~\ref{table}.
As in the proof of Theorem~\ref{main2}, we may use Tutte's theorem to work with flows in any abelian group of order~$k$, instead of 
with integer $k$-flows. We will do this implicitly throughout this section. 

\begin{proof}[Proof of Corollary \ref{planarcor}:] Let $G$ be a simple planar graph and let $G^*$ be the dual of $G$.
%Orient the edges of $G$ arbitrarily, and then orient the edges of $G^*$ consistently according to the rule indicated in the figure below.
%
%\begin{figure}[htp]
%\centerline{\includegraphics[width=3cm]{dualorient.pdf}}
%\end {figure}
%
Since $G$ is simple, $G^*$ is 3-edge-connected.
By Theorem~\ref{main1} we may choose a flow $\phi : E(G^*) \rightarrow \mathbb{Z}_3$ with $| \supp(\phi)| \ge \frac{5}{6} |E(G^*)|$. Let $H^*$ be the graph obtained from $G^*$ by contracting the edges with $\phi$-value zero. Now the restriction of $\phi$ to $E(H^*)$ is a nowhere-zero 3-flow. It follows by Theorem~\ref{th:Tutte-col} that the dual $H$ of $H^*$ has a proper 3-coloring. Since $H$ is obtained from $G$ by deleting the edges corresponding to $\phi^{-1}(0)$, the result holds.
\end{proof}

\begin{proof}[Proof of Theorem~\ref{2ec3flow}:] By Seymour's 6-flow theorem we may choose a nowhere-zero flow $\phi : E(G) \rightarrow \mathbb{Z}_6$ of $G$. Over all such possible choices for $\phi$, we claim that a fixed edge $e_0\in E(G)$ receives each of the five flow values equally often. To see this, recall that the number of nowhere zero $\mathbb{Z}_6$-flows in $G$ can be calculated via a contraction-deletion formula, by successively choosing edges until all that remains are loops and bridges. The same procedure can be used to count the number of nowhere zero $\mathbb{Z}_6$-flows with the property that $e_0$ receives a fixed flow value of $t\neq 0$. The process of contraction and deletion does not depend on this value $t$. Moreover, the number of nowhere-zero flows (where $e_0$ has flow value $t$) in each of these terminal graphs is independent of $t$. So, we indeed get our claim. In particular, each edge $e\in E(G)$ receives the value 3 in exactly one fifth of the possible choices for $\phi$. This means that on average one fifth of the edges in $G$ receive 3, and we may choose a particular $\phi_1$ so that no more than one fifth of the edges in $G$ receive 3. Define $\phi_2 : E(G) \rightarrow \mathbb{Z}_3$ by reducing each flow value in $\phi_1$  modulo 3. Note that $\phi_2$ is indeed a flow on $G$, and its only zero edges correspond to edges receiving 3 under $\phi_1$.
\end{proof}

Before we get to prove our result on approximative 4-flows, we need a version of a theorem of
Kaiser, Kr\'al' and Norine~\cite{KKN2006}.

\begin{theorem}[Kaiser, Kr\'al', Norine~\cite{KKN2006}]\label{th:KKN}
  Let $G$ be a 2-edge-connected cubic graph and $w : E(G) \to [0, \infty)$ a weighting function on the edges.
  Then~$G$ contains two perfect matchings $M_1$,~$M_2$ such that $w(M_1 \cap M_2) \le w(E(G))/15$.
\end{theorem}

\def\scal#1#2{\langle #1, #2\rangle}
\begin{proof}
  We follow closely the proof of Theorem~1 in~\cite{KKN2006}. We also refer the reader there for more
  background on the perfect matching polytope theorem. First, we define $\mu_1(e) = 1/3$ for every edge~$e$ of~$G$
  and observe that $\mu_1$ is in the perfect matching polytope (here we use the fact that $G$~has no bridge). 
  That means that $\mu_1$ is a convex combination of $\chi^{N_i}$ for some perfect matchings~$N_1$, $N_2$, \dots,~$N_l$.
  It follows that we can find $M_1$ such that $w(M_1) = \scal{w}{\chi^{M_1}} \le \scal{w}{\mu_1} = w(E(G))/3$. 
  (We use $\scal fg$ to denote the scalar product of~$f$ with~$g$ and $\chi^{M_1}$ for the characteristic function of~$M_1$.)

  Still following~\cite{KKN2006}, we choose $\mu_2(e) = 1/5$ if $e \in M_1$ and $\mu_2(e) = 2/5$ otherwise.
  We show now, that $\mu_2$ is also in the perfect matching polytope: We need
  to verify that $\mu_2(X) = \sum_{e \in X} \mu_2(e) \ge 1$ for every odd edge cut~$X$.
  If $|X|\ge 5$, this is obviously true. As $G$ is bridgeless, we only need to consider the case $|X|=3$.
  We recall how $M_1$ was chosen. It is one of the perfect matchings~$N_i$ that have $\mu_1$ as their convex combination.
  As $\mu_1(X) = 1$ (by definition) and $\chi^{N_i}(X) = |N_i \cap X| \ge 1$ for every~$i$ (as $N_i$ is a perfect matching and $X$~is an odd cut),
  it follows that $|N_i \cap X| = 1$ for every~$i$. Therefore, $\mu_2(X) = \frac 15 + \frac 25 + \frac 25 = 1$.

  We define $w_1(e) = w(e)\chi^{M_1}(e)$. As $\mu_2$ is a convex combination of characteristic functions of perfect matchings, there
  is a perfect matching~$M_2$, such that
  $w_1(M_2) = \scal{w_1}{\chi^{M_2}} \le \scal{w_1}{\mu_2} = w_1(E(G))/5 \le w(M_1)/5 \le w(E(G))/15$,
  as claimed.
\end{proof}

Next we prove another theorem claimed in our introduction.  For a graph~$G$ and a pair of edges $e,f$ which are both incident with the vertex $v$, we
  \emph{lift} $e$ and $f$ by creating a new vertex $x$ and changing $e,f$ to have $x$ as an endpoint instead of $v$.

\begin{theorem}\label{th:4flow}
  Every 2-edge-connected graph $G$ has a 4-flow $\phi$ with $| \supp(\phi) | \ge \frac{14}{15} |E(G)|$.
\end{theorem}

\begin{proof} We may assume that $G$ is not Eulerian, since in this case it has a 2-flow with support $E(G)$.  It follows from Mader's Splitting Theorem~\cite{Mader78} that we may repeatedly perform the aforementioned lift operation to obtain a graph $G'$ which is subcubic and still cyclically 3-edge-connected.  Let $G''$ be the graph obtained from $G'$ by suppressing every vertex of degree 2.  Then each edge $e \in E(G'')$ corresponds to a path $P_e$ of $G'$ and we let $w(e)$ be the length of this path.

It follows from Theorem~\ref{th:KKN}, that there exists a pair of perfect matchings $M_1, M_2$ in $G''$ so that
$w(M_1 \cap M_2) \le \frac{1}{15} w(E(G))$.  For $i=1,2$ define the function $\phi_i : E(G') \rightarrow \mathbb{Z}_2$ by the rule that $\phi_i (f) = 0$ if the edge of $G''$ associated with $f$ is in $M_i$ and otherwise $\phi_i(f) = 1$.  Now $\phi_1 \times \phi_2$ is a $\mathbb{Z}_2 \times \mathbb{Z}_2$-flow of $G$ with support of size $\ge \frac{14}{15} |E(G)|$ as desired.
\end{proof}

Note that the bound in Theorem~\ref{th:4flow} is achieved by the Petersen graph.
\begin{theorem}
For every $j\geq 1$, $h(2j, 2) = h(2j+1, 2) = \frac{2j}{2j+1}$.
\end{theorem}

\begin{proof} For every $j \ge 1$ there exists a $(2j+1)$-regular graph $H$ which is $(2j+1)$-edge-connected.  If $\phi$ is a 2-flow of $H$, then for every vertex $v$ at least one of the $2j+1$ edges incident with $v$ will not be contained in $\supp(\phi)$.  It follows that $| \supp(\phi) | \le \frac{2j}{2j+1} |E(H)|$ thus giving the bound $h(2j+1,2) \le \frac{2j}{2j+1}$.

For the other direction, let $G$ be an arbitrary $2j$-edge-connected graph.  We may assume that $G$ is not Eulerian, since in this case $G$ has a 2-flow with support
  $E(G)$.  As in the previous proof, we repeatedly apply Mader's Splitting Theorem to vertices with even degree $\ge 4$ and to vertices with odd degree $\ge 2j+3$, and we
  let $G'$ be the resulting graph.  Let $G''$ be the graph obtained from $G'$ by suppressing vertices of degree 2; so every $e \in E(G'')$ corresponds to a path $P_e$ of
  $G'$ and we let $w(e)$ be the length of this path.

It follows from Edmond's Matching Polytope theorem that there exists a list $M_1, \ldots, M_{(2j+1)t}$ of perfect matchings in $G''$ so that every edge of $G''$ is contained in exactly $t$ of these matchings.  So, we may choose $1 \le k \le (2j+1)t$ so that $w(M_k) \le \frac{1}{2j+1} \sum_{e \in E(G'')} w(e) = \frac{1}{2j+1} |E(G)|$.  Now in the original graph $G$ there is a 2-flow $\phi$ with support $E(G) \setminus \cup_{e \in M_j} E(P_e)$ and thus $| \supp(\phi)| \ge \frac{2j}{2j+1}|E(G)|$.
\end{proof}

\section{Ears}

Although Theorem~\ref{main2}, which we wish to prove, concerns 3-edge-connected graphs, our proof will involve a reductive process which encounters graphs which are only 2-edge-connected.  In preparation for this, we will establish some terminology and tools for working with 2-edge-connected graphs.

\subsection{Ear Decomposition}

A well-known basic result in graph theory is the ear decomposition theorem which asserts that every 2-connected graph has a certain recursive structure.  For our purposes we will want to work with edge-connectivity, and it will be helpful to recast this basic concept in this alternate setting.  Since this notion is extremely close to the original, we will adopt the same terminology.  Accordingly, we now define an \emph{ear} of a graph $G$ to be a subgraph $P \subseteq G$ which satisfies one of the following:
\begin{itemize}
  \item $P$ is a nontrivial path (i.e., a path with at least two vertices),
	all interior vertices of $P$ have degree $2$ in $G$, but both endpoints have degree at least $3$ in $G$.
\item $P$ is a cycle of $G$ containing exactly one vertex with degree $\neq 2$ in $G$.
\item $P = G$ is a cycle.
\end{itemize}

For an arbitrary graph $H$ we let $H^{\times}$ denote the graph obtained from $H$ by deleting all isolated vertices (i.e., vertices of degree 0).  If $P$ is an ear of $G$, then \emph{removing} $P$ brings us to the new graph $(G - E(P))^{\times}$.   We define a \emph{partial ear decomposition} of a graph $G$ to be a list $P_1, P_2, \ldots, P_{\ell}$ of subgraphs of $G$ satisfying the following:
\begin{enumerate}
\item $E(P_i) \cap E(P_j) = \emptyset$ whenever $i \neq j$.
\item $P_j$ is an ear of the graph obtained from $G$ by removing $P_{\ell}, \ldots, P_{j+1}$.
\end{enumerate}
A partial ear decomposition $P_1, \ldots, P_{\ell}$ is called a \emph{full ear decomposition} if $E(G) = \cup_{i=1}^{\ell} E(P_i)$.  Note that in this case the first graph $P_1$ must be a cycle.  If $P_1, \ldots, P_{\ell}$ is a full ear decomposition of $G$, then a straightforward induction implies that each of the graphs  $P_1 \cup \ldots \cup P_j$  will be 2-edge-connected, so in particular, any graph with a full ear-decomposition must be 2-edge-connected.  Conversely, for any 2-edge-connected graph $G$ we may construct a full ear-decomposition greedily starting from an arbitrary cycle (if we have chosen $P_1, \ldots, P_j$ and there is an edge $e \not\in \cup_{i=1}^j E(P_i)$, then by Menger's Theorem there are two edge-disjoint paths starting at the ends of $e$ and ending in $\cup_{i=1}^j V(P_i)$ and the union of these paths together with $e$ contains a suitable choice for $P_{j+1}$).  This yields the following basic property.

\begin{proposition}
A graph is 2-edge-connected if and only if it has a full ear decomposition.
\end{proposition}

\subsection{Weighted Graphs and Ear Labellings}

We define a \emph{weighted graph} to be a graph $G$ equipped with a function $\mu_G : V(G) \rightarrow \mathbb{Z}_3$, and we call $G$ a \emph{zero-sum} weighted graph if $\mu_G$ is zero-sum.  In preparation for the proof of Theorem~\ref{main2}, we now introduce a framework to move from one weighted graph to another by removing ears.

For our main theorem we consider an oriented zero-sum weighted graph $G$ and we are interested in finding a function $\phi : E(G) \rightarrow \mathbb{Z}_3$ with boundary
$\mu_G$ and large support.  Let us take a moment to consider the possible behaviours of such a function $\phi$ on an ear.  So, let $P$ be an ear of $G$ and express $P$ as
either a path or closed path with vertex-edge sequence $v_1, e_1, v_2, \ldots, e_m, v_{m+1}$ so that $v_2, \ldots, v_m$ have degree 2 in $G$ (i.e., if $P$ is a cycle containing a vertex of $G$ with degree $\ge 3$, then this vertex is $v_1 = v_{m+1}$).  Assume further (for simplicity) that every edge $e_i$ is oriented from $v_i$ to $v_{i+1}$.  If we have chosen the value $\phi (e_1)$ and we wish for our function $\phi$ to satisfy $\partial \phi (v_2) = \mu_G(v_2)$, then (assuming $m \ge 2$) we must assign $\phi (e_2) = \phi(e_1) + \mu_G(v_2)$.  This in turn forces the value of $\phi(e_3)$ and all the remaining edges of $P$ (in general, $\phi (e_j) = \phi(e_1) + \sum_{i=2}^{j} \mu_G (v_i)$).  So, when constructing our function $\phi$ we have just a single degree of freedom for each ear.  These choices will be significant for us, so we will introduce a little terminology to work with them.  If $P$ is an ear of $G$, a function $\psi : E(P) \rightarrow \mathbb{Z}_3$ is called an \emph{ear labelling} if $\partial \psi(v) = \mu_G(v)$ for every vertex $v$ which is in the interior of the path $P$.   The following observation is a straightforward consequence of this discussion (together with the basic fact that when $P = G$ is a cycle, every function $\psi : E(P) \rightarrow \mathbb{Z}_3$ for which $\partial \psi$ and $\mu_G$ agree on all but one vertex will satisfy $\partial \psi = \mu_G$, since these two functions are both zero-sum).

\begin{observation}
Let $P$ be an ear of the oriented zero-sum weighted graph $G$.  Then there are exactly three distinct ear labellings $\psi_1, \psi_2, \psi_3$ of $P$,
and every $e \in E(P)$ has value 0 in exactly one of these labellings.
In fact, in the above-defined orientation we may assume $\psi_{i+1} = \psi_i + 1$ (indices modulo~3).
\end{observation}

Since we are looking to construct functions $\phi : E(G) \rightarrow \mathbb{Z}_3$ which have large support, we will naturally be interested in ears $P$
of $G$ which have an ear labelling with large support.  If $\psi_1, \psi_2, \psi_3$ are the ear labellings of $P$, then by the above discussion, the
average size of the support of an ear labelling~$\psi_i$ will be precisely $\frac{2}{3} |E(P)|$.  When $|E(P)|$ is a multiple of 3, we may have
$| \supp ( \psi_i) | = \frac{2}{3} |E(P)|$ for $1 \le i \le 3$.
(Equivalently, still assuming the ``forward orientation'' of~$P$, we may have $|\psi_1^{-1}(0)| = |\psi_1^{-1}(1)| = |\psi_1^{-1}(2)|$.)
In this extreme case we say that $P$ is \emph{equitable}, and in all other cases we call $P$ \emph{inequitable}.  When $|E(P)|$ is not a multiple of 3 (or more generally when $P$ is inequitable), there exists at least one ear labelling $\psi_i$ with $| \supp (\psi_i) | > \frac{2}{3} |E(P)|$ and our proof will frequently exploit this.  Indeed, the key to our argument is getting a small advantage for each inequitable ear.

In preparation for this we now introduce a general definition.  If $H$ is a subgraph of $G$ and $\psi : E(H) \rightarrow \mathbb{Z}_3$, we define the \emph{gain} of $\psi$ to be
\[ \gain(\psi) = 24 | \supp(\psi)| - 16 |E(H)|.\]
The following lemma gives our basic tool for finding good ear labellings.  We will use this extensively in the remainder of the paper.

\begin{lemma}
\label{goodgain}
Let $G$ be a weighted graph. If $P$ is an inequitable ear of $G$ such that $|E(P)| = 3k + i$ where $1 \le i \le 3$, then $P$ has
an ear labelling $\psi : E(P) \rightarrow \mathbb{Z}_3$ with $\gain(\psi) \ge 8i$.
If $P$ is equitable, then it has an ear labelling~$\psi$ with $\gain(\psi)=0$.
\end{lemma}

\begin{proof} Choose an ear labelling $\psi$ of $P$ for which $|\supp(\psi)|$ is maximum, and note that our assumptions imply that $|\psi^{-1}(0)|\leq k$, i.e.,  $| \supp(\psi)| \ge 2k+i$.  This gives
$\gain(\psi) \ge 24(2k+i) - 16(3k+i) = 8i$ as desired.
\end{proof}

\bigskip

Next we introduce some terminology to facilitate the process of deciding on a particular ear labelling, and then removing this ear from the weighted
graph.  If $P$ is an ear of $G$ and $\psi : E(P) \rightarrow \mathbb{Z}_3$ is an ear-labelling, we define the $\psi$-\emph{removal} of $P$ (from~$G$)
to be the weighted graph $G' = \left( G \setminus E(P) \right)^{\times}$ equipped with the weight function $\mu_{G'} : V(G') \rightarrow \mathbb{Z}_3$
given by the following rule (we set $\partial \psi(v) = 0$ if $v \not\in V(P)$)
\[ \mu_{G'}(v) = \mu_G(v) - \partial \psi (v). \]
Since $\mu_G$ and $\partial \psi$ both sum to zero, the same holds for the function $\mu_G -  \partial \psi$.  %However, this latter function is identically zero on the interior vertices of the path $P$.
It follows that the function $\mu_{G'}$ will be zero-sum.  The following straightforward observation shows that we can combine a function $\phi' : E(G') \rightarrow \mathbb{Z}_3$ with boundary $\mu_{G'}$ with our ear labelling $\psi$ to obtain a function on $E(G)$ with boundary $\mu_G$.

\begin{observation}
\label{putback}
Let $G$ be an oriented zero-sum weighted graph, let $P \subseteq G$ be an ear, let $\psi$ be an ear labelling of $P$, and let $G'$ be the $\psi$-removal of $P$.  If $\phi' : E(G') \rightarrow \mathbb{Z}_3$ satisfies $\partial \phi' = \mu_{G'}$, then the following function $\phi : E(G) \rightarrow \mathbb{Z}_3$ has $\partial \phi = \mu_G$:
\[ \phi(e) =	\left\{ \begin{array}{ll}	
	\phi'(e)	&	\mbox{if $e \in E(G')$}	\\
	\psi(e)	&	\mbox{if $e \in E(P)$.}
	\end{array} \right. \]
\end{observation}

\section{Setup}

In this section we will state our workhorse lemma, and use it to prove our main theorem.  Then we will set the stage for our proof of this lemma by fixing a minimal counterexample to it, and establishing some initial properties of this weighted graph.

\subsection{Framework}

Before we are ready to state our main lemma (Lemma~\ref{workhorse} below), let us pause to introduce the type of connectivity we will be working with.  Throughout the heart of the proof of the central
lemma we will work with graphs $G$ which are subdivisions of 3-edge-connected graphs.  Since we are permitting loops, and any 1-vertex graph is 3-edge-connected, it is
possible for $G$ to be a cycle or more generally a collection of cycles which intersect at a common vertex.  So, $G$ is a subdivision of a 3-edge-connected graph if and
only if $G$ is a 2-edge-connected graph which is cyclically 3-edge-connected (i.e., if $S$ is an edge-cut which separates cycles, then $|S| \ge 3$).

As seen in the previous section, ears with different lengths modulo 3 will behave differently when constructing our desired function.  To deal with this behaviour we will introduce a bonus function which assigns to each ear a value which indicates in some sense the amount we expect to gain from it.  For an ear $P$ we define the \emph{bonus} of $P$ as follows:
\[ \bonus(P) = \left\{ \begin{array}{cl}
		0 & \mbox{if $P$ is equitable} \\
		3 & \mbox{if $|E(P)| \equiv 2 \pmod{3}$.} \\
		4 & \mbox{otherwise}
		\end{array} \right. \]
For a subgraph $H \subseteq G$ which is a union of disjoint ears, that is $H = \cup_{i=1}^{\ell} P_i$ (where
each $P_i$ is an ear in~$G$) we define $\bonus(H) = \sum_{i=1}^{\ell} \bonus(P_i)$.
(We warn the reader of a possible confusion: the $P_i$'s do not necessary form an ear decomposition of~$H$,
and they may also not be ears in~$H$, but only ears in~$G$.)
So $\bonus(G)$ is the sum of the bonuses of all of the ears of $G$.  With this terminology in place, we are finally ready to state the workhorse lemma which will imply Theorem~\ref{main2}.

\begin{lemma}
\label{workhorse}
Let $G$ be an oriented zero-sum weighted graph and assume that $G$ is a subdivision of a 3-edge-connected graph.  Then there exists $\phi : E(G) \rightarrow \mathbb{Z}_3$ satisfying:
\begin{itemize}
\item $ \partial \phi = \mu_G$,
\item $\gain( \phi ) \ge \bonus(G)$.
\end{itemize}
\end{lemma}

Now let us see that this lemma implies our main theorem.

\begin{proof}[Proof of Theorem~\ref{main2}:]
Let $G$ be a 3-edge-connected graph and let $\mu : V(G) \rightarrow \mathbb{Z}_3$ be zero-sum.  Every edge of $G$ is an ear with length 1 mod 3.  So, the lemma gives us a function $\phi : E(G) \rightarrow \mathbb{Z}_3$ with $\partial \phi = \mu$ and $| \supp (\phi) | \ge \frac{2}{3} |E(G)| + \frac{1}{24} \bonus(G) = \frac{5}{6} |E(G)|$, as desired.
\end{proof}

We try now to preview the main ideas of the proof, before we get into the lengthy details.
It was crucial to find the proper setting: we have generalized Theorem~\ref{main1}
to Theorem~\ref{main2} (about flows with a given boundary). Then we have extended it even further
by defining an appropriate bonus system: Lemma~\ref{workhorse} provides a result about a richer class of graphs.
Building upon this choice of graphs, we will be able to find many types of reductions to a smaller graph,
while staying in the class.
These reductions involve deleting edges or vertices of a graph (Lemma~\ref{removingedge} and~\ref{removingvertex})
or even pairs of adjacent vertices (Observation~\ref{obs4cuts}).

\subsection{Minimal Counterexample}

Assume (for a contradiction) that Lemma~\ref{workhorse} is false, and choose a counterexample $G$ for which $|E(G)|$ is minimum.
We will spend the rest of the paper discussing properties of~$G$, building towards a contradiction.
We start with two very straightforward lemmas concerning $G$. The first one shows that $G$ is not too basic in structure, the second one shows it has no equitable ear.

\begin{lemma}
\label{2deg3}
$G$ has at least two vertices with degree at least $3$.
\end{lemma}

\begin{proof}
Suppose (for a contradiction) that $G$ has at most one vertex of degree at least $3$.
Apply Lemma~\ref{goodgain} to choose an ear labelling $\psi : E(P) \rightarrow \mathbb{Z}_3$ with $\gain(\psi) \ge \bonus(P)$ for every ear $P$ of $G$.  Let $\phi: E(G) \rightarrow \mathbb{Z}_3$ be the union of these functions and note that $\gain(\phi) \ge \bonus(G)$.  It follows from our construction that $\partial \phi (v) = \mu_G(v)$ holds at every vertex $v$ with $\deg(v) = 2$.  Since $\partial \phi$ and $\mu_G$ are both zero-sum, we conclude that $\partial \phi = \mu_G$, which is a contradiction. Hence $G$ has at least two vertices with degree at least $3$.
\end{proof}

If $H$ is a weighted graph and $uv \in E(H)$, then \emph{contracting} $uv$ gives a new weighted graph $H'$ where the underlying graph $H'$ is obtained from $H$ by contracting the edge $uv$ to form a new vertex $y$, and the weight function $\mu_{H'}$ is given by the following rule:
\[ \mu_{H'} (x) = \left\{ \begin{array}{cl}
		\mu_{H}(x)	&	\mbox{if $x \neq y$}	\\
		\mu_H(u) + \mu_H(v)	&	\mbox{if $x=y$}
		\end{array} \right. \]
Note that if $\mu_H$ is zero-sum, then $\mu_{H'}$ will also be zero-sum.

\begin{lemma}
\label{noequitable}
If $P$ is an ear of $G$ with ear labellings $\psi_1, \psi_2, \psi_3$, then there do not exist $e_1, e_2, e_3 \in E(P)$ so that
$\psi_i(e_i) = 0$ for $1 \le i \le 3$.  In particular, $G$ has no equitable ear.
\end{lemma}

\begin{proof}
  Suppose to the contrary that there exist such edges $e_1, e_2, e_3$. Let $G_1$ be the weighted graph obtained from~$G$ by contracting~$e_1$, let $G_2$ be obtained
from~$G_1$ by contracting~$e_2$, and let $G_3$ be obtained from~$G_2$ by contracting~$e_3$. Now $E(P)\setminus \{e_1,e_2,e_3\}$ is either empty (in which case $P$ was
equitable and had bonus~0) or it induces an ear of $G_3$ with the same length modulo 3 as~$P$.
It follows from our definitions that $\bonus(G_3) = \bonus(G)$, so by the minimality of
our counterexample, we may choose a function $\phi : E(G_3) \rightarrow \mathbb{Z}_3$ with $\partial \phi = \mu_{G_3}$ and $\gain(\phi) \ge \bonus(G)$.  Now we will
step back from~$G_3$ to~$G_2$ by reversing the contraction of $e_3$.  Since any two zero-sum functions which agree on all but one vertex also agree on the last, we may
extend $\phi$ by choosing a value for $\phi(e_3)$ in such a way that $\partial \phi = \mu_{G_2}$.  Repeating this argument to reverse the contraction of $e_2$ and then
$e_1$ results in a function $\phi : E(G) \rightarrow \mathbb{Z}_3$ with $\partial \phi = \mu_G$.  The restriction of $\phi$ to $E(P)$ is an ear labelling, so by our
assumption $\phi(e_i) = 0$ for exactly one $1 \le i \le 3$.  This function $\phi$ satisfies the conclusion of Lemma~\ref{workhorse}, thus giving us a contradiction.
Therefore no such ear may exist.  In particular this implies that $G$ has no equitable ear.
\end{proof}

\subsection{Contraction}

In order to handle more general situations, we will now establish some tools for finding subgraphs of $G$ which we can contract.  In many of our arguments, we will be
interested in removing a sequence of ears which have nonzero length modulo 3, and we will call on Lemma~\ref{goodgain} to find suitable ear labellings.  In light of this,
it is natural (and helpful) to introduce a concept of gain for a partial ear decomposition.  Let $H$ be a graph, let $P_1, P_2, \ldots, P_{\ell}$ be a partial ear
decomposition of $H$, and assume that $|E(P_i)| \equiv r_i \pmod{3}$ where $0 \le r_i \le 2$ for every $1 \le i \le \ell$.
Then we define (motivated by Lemma~\ref{goodgain})  the \emph{gain} of $P_1, P_2,
\ldots, P_{\ell}$ to be $\sum_{i = 1}^{\ell} 8 r_i$.  The following easy lemma gives a first use of this concept, explaining the connection to the gain of an ear
labelling.

\begin{lemma}
\label{fulldecompose}
Let $H$ be an oriented zero-sum weighted graph with a full ear decomposition $P_1, \ldots, P_{\ell}$.  Then there exists $\phi : E(H) \rightarrow \mathbb{Z}_3$ with $\partial \phi = \mu_H$ so that $\gain(\phi)$ is at least the gain of $P_1, \ldots, P_{\ell}$.
\end{lemma}

\begin{proof} Starting with $i = \ell$ and working down to $i=1$ we apply Lemma~\ref{goodgain} to choose an ear labelling $\psi_i$ of $P_i$ and then replace $H$ by the $\psi_i$-removal of $P_i$.  If $\phi$ is the union of these $\psi_i$ functions, then by repeated applications of Lemma~\ref{putback}, we have $\partial \phi = \mu_H$ and by construction we have that $\gain(\phi)$ is at least the gain of $P_1, \ldots, P_{\ell}$.
\end{proof}

\bigskip

Next we define a notion of contractible based on this concept of  gain.

\begin{definition}
Let $H \subseteq G$ be a union of ears.  We say that $H$ is \emph{contractible} if there exists a full ear decomposition $P_1, P_2, \ldots, P_{\ell}$ of $H$ with gain
at least~$\bonus(H)$.
\end{definition}

For clarity, we note that $\bonus(H) = \sum_{j=1}^k \bonus(Q_j)$, where $Q_1$, \dots, $Q_k$ are the ears of~$G$ that
comprise~$H$. However, the $P_i$'s and the $Q_j$'s are, in general, different (each of the $P_i$'s is a union of some of the $Q_j$'s).

Also note that, in particular, a contractible subgraph is 2-edge-connected.
The next lemma shows that this definition captures the desired notion.

\begin{lemma}
$G$ has no contractible subgraph.
\end{lemma}

\begin{proof}
Suppose (for a contradiction) that $H$ is a contractible subgraph of $G$ and that $P_1, \ldots, P_{\ell}$ is the indicated ear decomposition of $H$.  Note that $H$ must be
  connected since it has a full ear decomposition.  Let $G'$ be the weighted graph obtained from $G$ by contracting all of the edges in $H$.  By the minimality of our
  counterexample we may choose a function $\phi' : E(G') \rightarrow \mathbb{Z}_3$ with $\partial \phi' = \mu_{G'}$ and
  $\gain(\phi') \ge \bonus(G') = \bonus(G) - \bonus(H)$.  Extend $\phi'$ to have domain $E(G)$ by setting $\phi'(e) = 0$ for every $e \in E(H)$.  Now $\partial \phi'$ and $\mu_G$ agree on every vertex in $V(G)
  \setminus V(H)$. We define $\nu : V(H) \rightarrow \mathbb{Z}_3$ by the rule $\nu(v) = \mu_G(v) - \partial \phi'(v)$.  Since $\mu_G$ and $\partial \phi'$ are both
  zero-sum, the function $\nu$ will also be zero-sum.  Using Lemma~\ref{fulldecompose} we may choose a function $\psi : E(H) \rightarrow \mathbb{Z}_3$ with $\partial \psi
  = \nu$ and $\gain(\psi) \ge \bonus(H)$.  Now the union of $\phi'$ and $\psi$ has boundary $\mu_G$ and gain at least $\bonus(G)$, which is a contradiction.
\end{proof}

\bigskip

Next we call on the concept of contractible to eliminate some additional subgraphs of $G$.

\begin{lemma}
\label{easycont}
There does not exist a cycle $C \subseteq G$ satisfying one of
\begin{enumerate}
\item $C$ is the union of at most two ears and $|E(C)| \not\equiv 0 \pmod{3}$,
\item $C$ is the union of three or four ears and $|E(C)| \equiv 2 \pmod{3}$.
\end{enumerate}
\end{lemma}

\begin{proof}
In the former case, $\bonus(C) \le 8$ and $C$ has a full ear decomposition (consisting of the entire cycle $C$) of gain $\ge 8$, so $C$ is contractible. In the latter case, $\bonus(C) \le 16$ and $C$ has a full ear decomposition of gain $\ge 16$, so it is contractible.
\end{proof}

\subsection{Deletion}

Next we will introduce some tools to facilitate deletion arguments.  Unlike contraction which automatically preserves our desired edge-connectivity, we will need to be careful when deleting.

\begin{definition}
A subgraph $H \subseteq G$ is \emph{reducible} if there exists a partial ear decomposition $P_1, \ldots, P_{\ell}$ of $G$ together with functions $\psi_i : E(P_i) \rightarrow \mathbb{Z}_3$ for $1 \le i \le \ell$ satisfying the following properties:
\begin{itemize}
\item $H = \cup_{i=1}^{\ell} P_i$,
\item $(G - E(H))^{\times}$ is empty or is a subdivision of a 3-edge-connected graph,
%\item The function $\psi_j$ is an ear labelling of $P_j$ in the weighted graph that is the $\psi_{\ell}$-removal of $P_{\ell}$ from $\ldots$ from the
%  $\psi_{j+1}$-removal of $P_{j+1}$ from~$G$.
\item The function $\psi_j$ is an ear labelling of $P_j$ in the $\psi_{j+1}$-removal of $P_{j+1}$ from $\ldots$ from the
  $\psi_{\ell}$-removal of $P_{\ell}$ from~$G$.
\item The weighted graph $G'$ that is
  the $\psi_{1}$-removal of $P_{1}$ from $\ldots$ from the $\psi_{\ell}$-removal of $P_\ell$ from~$G$
	satisfies
  $\sum_{i=1}^{\ell} \gain(\psi_i) \ge \bonus(G) - \bonus(G')$.
\end{itemize}
\end{definition}

The following observation shows that our minimal counterexample $G$ cannot contain a subgraph of this type.

\begin{observation}
\label{noreducible}
$G$ does not have a reducible subgraph $H$.
\end{observation}

\begin{proof}
  Suppose (for a contradiction) that such a subgraph exists and let $P_1, \ldots, P_{\ell}$ and $\psi_1, \ldots, \psi_{\ell}$ and $G'$ be as in the above
definition.  By the minimality of the counterexample $G$ (and by the second property in the definition of reducible) we may choose a function $\phi' : E(G') \rightarrow \mathbb{Z}_3$ with $\partial \phi' = \mu_{G'}$ and $\gain(\phi') \ge \bonus(G')$.  Now repeatedly applying Observation~\ref{putback} to $P_1, \ldots, P_{\ell}$
gives us a function $\phi : E(G) \rightarrow \mathbb{Z}_3$ with $\partial \phi = \mu_G$ and $\gain(\phi) \ge \bonus(G)$ which is a contradiction.
\end{proof}

\bigskip

We will apply the above observation repeatedly, but we will begin with an easy instance.

\begin{lemma}
\label{easydelconseq}
$G$ does not have either of the following:
\begin{itemize}
\item a cycle which is an ear,
\item a cycle consisting of two ears containing a vertex of degree 3 in $G$.
\end{itemize}
\end{lemma}

\begin{proof}
If $G$ has a cycle $C$ which is an ear, then Lemma~\ref{2deg3} (and the assumption that $G$ is a subdivision of a 3-edge-connected graph) implies that
  $C$ must contain a vertex of degree at least $5$ in $G$. Thus, removing $C$ does not merge two ears of~$G$ into one,
  implying $\bonus(G) = \bonus( (G-E(C))^\times ) + \bonus(C)$.
  Lemma~\ref{goodgain} allows us to choose an ear labelling $\psi$ of $C$ with $\gain(\psi)
  \ge \bonus(C)$ and it then follows that $C$ is a reducible subgraph (contradicting Observation~\ref{noreducible}).

Next suppose that $G$ has a cycle $C$ consisting of two ears $P,P'$ so that an endpoint $x$ of $P$ and $P'$ has $\deg_G(x) = 3$.  Let $y$ be the other endpoint of these two ears, and let $P''$ be the third ear incident to $x$. By Lemma~\ref{easycont} we have $|E(C)| \equiv 0 \pmod{3}$, so we may assume (without loss) that $|E(P)| \not\equiv 1 \pmod{3}$.  Therefore, by Lemma~\ref{goodgain} (and Lemma~\ref{noequitable}) we may choose an ear labelling~$\psi$ of $P$ with gain at least 16. If $G'$ is the $\psi$-removal of $P$, then we have that $G'$ is a subdivision of a 3-edge-connected graph, as the ears $P'$ and $P''$ merge in $G'$. If $y$ has degree at least 4 in $G$, then there are just three ears of $G$ which are not ears of $G'$. If $y$ has degree 3 in $G$, then since $G$ is a subdivision of a 3-edge-connected graph, $P''$ must be incident to $y$, and $G$ must consist of only these three ears. Hence, in either case, $\bonus(G') \ge \bonus(G) - 12$. Thus $P$ is reducible, giving us a contradiction.
\end{proof}

\subsection{Pushing a 3-Edge Cut}

We have introduced the concept of a reducible subgraph and this will be key in our proof.  However, in order for a subgraph $H$ to be reducible, it must be the case that $(G - E(H))^{\times}$ is a subdivision of a 3-edge-connected graph (or it is empty).  So, in our search for reducible subgraphs, we will need to have some tools for finding subgraphs which we can delete so as to maintain our connectivity.  This section gives the first of these tools.  It is based on a simple minimality property for 3-edge-cuts.  Once we have this  in place, we will apply it to construct a new graph $G^{\bullet}$ which will be convenient for our operations.

For any subset of vertices $X$ of a graph, we let $d(X)$ denote the number of edges with exactly one endpoint in $X$.  The proof of our main tool from this section is
based on uncrossing arguments that call on the following inequalities which hold for any two subsets $X,Y$ of vertices.
\begin{align}
d(X \cap Y) + d(X \cup Y) &\le d(X) + d(Y) 	 \label{submod} \\
d(X \setminus Y) + d(Y \setminus X) &\le d(X) + d(Y) \label{other}
\end{align}

\begin{lemma}
\label{3edgecut}
Let $H = (V,E) $ be a 3-edge-connected graph and let $X \subseteq V$ be minimal subject to the following conditions
\begin{enumerate}
\item $d(X) = 3$,
\item $X$ induces a graph containing at least two cycles.
\end{enumerate}
If $e \in E$ has both ends in $X$, then every 3-edge-cut of $H$ containing $e$ has the form $\delta(Y)$ where $Y \subseteq X$ induces a graph containing at most one cycle.
\end{lemma}

\begin{proof} Suppose that $\delta(Y)$ is a 3-edge-cut of $H$ with $e \in \delta(Y)$, and consider the Venn diagram on $V$ given by the sets $X$ and $Y$ as depicted in the figure below.  By possibly swapping $Y$ with its complement, we may assume that $V\setminus (X\cup Y)$ is nonempty. First suppose that $Y \setminus X = \emptyset$.  Then inequality~(\ref{submod}) gives us $d(X \cap Y) + d(X \cup Y) \le 6$, but then the 3-edge-connectivity of $H$ implies $d(X \cap Y) = 3 = d(X \cup Y)$.  It now follows from the minimality of $X$ that $Y = X \cap Y$ is a subset of $X$ which induces a graph containing at most one cycle, as desired.

\begin{figure}[htp]
\centerline{\includegraphics[width=4cm]{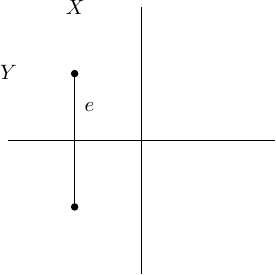}}
\end {figure}

Thus, we may assume that all four sets in our Venn Diagram are nonempty.
Now the inequalities~(\ref{submod}) and~(\ref{other}) imply that $d(X \cap Y) = d(X \cup Y) = d(X \setminus Y) = d(Y \setminus X) = 3$, but this is impossible.  (For instance $d(X \setminus Y) + d(X \cap Y) \equiv d(X) \pmod{2}$, so our present conditions violate parity.)
\end{proof}

\bigskip

In order to productively use this lemma, we will need to show that our graph $G$ is not one of a couple of small graphs.  One of these graphs is a two vertex graph with exactly three edges none of which is a loop, which we call \emph{Theta}.

\begin{lemma}
$G$ is not isomorphic to a subdivision of either $K_4$ or Theta.
\end{lemma}

\begin{proof}
By Lemma~\ref{easydelconseq}, $G$ is not a subdivision of Theta.
Now suppose (for a contradiction) that $G$ is isomorphic to a subdivision of $K_4$.  If there is an ear $P$ of $G$ with length 0 mod 3, then it has an ear labelling with gain at least 24 (by Lemma~\ref{goodgain} and~\ref{noequitable}), and thus $P$ is a reducible subgraph since $\bonus(G) \le 24$.
%If $G$ is a subdivision of Theta, then a similar argument shows that every ear of $G$ has length 1 mod 3, but then every full ear decomposition of $G$ has gain 24 and we have a contradiction to Lemma~\ref{fulldecompose}.  So, we may now assume that $G$ is a subdivision of $K_4$.
If $G$ has an ear $P$ with length 2 mod 3, then setting $P'$ to be the unique ear of~$G$ which does not have an end in common with~$P$, we may construct a full ear
decomposition using $P$ and $P'$ which has gain at least $24$ (again, Lemma~\ref{fulldecompose} shows that $G$~is not a counterexample).
So, we may assume that every ear of $G$ has length 1 mod 3.  Again in this case we find a full ear decomposition of $G$ with gain at least 24, thus giving us a contradiction.
\end{proof}

\bigskip

We need just one more simple property of $G$ before we can take advantage of Lemma~\ref{3edgecut}.

\begin{lemma}
$G$ has a 3-edge-cut $\delta(X)$ so that $X$ induces a subgraph containing at least two cycles.
\end{lemma}

\begin{proof}
First suppose that $G$ does not have an edge-cut of size 3.  In this case, Lemma~\ref{goodgain} implies that every ear is a reducible subgraph, so $G$ is not a minimal
counterexample.  Therefore, we may assume that $G$ has some edge cut $\delta(X)$ with $d(X) = 3$.  If neither $X$ nor $V(G) \setminus X$ satisfy the lemma, then each of
these sets induces a graph with at most one cycle.  If $X$ induces a graph with no cycles, then since $G$ is a subdivision of a 3-edge-connected graph and $d(X)=3$, $G[X]$ is a subdivision of $K_{1, 3}$.  If $X$ induces a graph with  a unique cycle $C$, then Lemma~\ref{easydelconseq} says that $C$ is not an ear nor it is composed of two ears with a vertex of degree three. So $C$ must be the union of
exactly three ears whose endpoints all have degree 3 in $G$ (otherwise more than three edges are incident to $C$, so $G[X]$ would have a second cycle or a small edge-cut).  A similar argument for $V(G) \setminus X$ shows that $G$ must be a subdivision of either Theta, $K_4$ or the graph Prism depicted below.  The
former two cases are impossible by the previous lemma; in the last case choosing $X$ to be the complement of a vertex of degree 3 yields a suitable edge-cut.
\end{proof}

\begin{figure}[htp]
\centerline{\includegraphics[width=2cm]{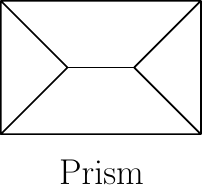}}
\end {figure}

\subsection[Creating $G^{\bullet}$]{Creating $\mathbf{G^{\bullet}}$}

With the above lemma in place we shall now choose a set $X \subseteq V(G)$ so that $d(X) = 3$ and the graph induced by $X$ contains at least two cycles, and subject to
  this we choose $X$ minimal. Let $P_1, P_2, P_3$ be the ears of $G$ containing an edge of $\delta(X)$ and define $W$ to be the set of all vertices which are not in $X$
  and are not interior vertices in $P_1$, $P_2$, or $P_3$.  We define $G^{\bullet}$ to be the weighted graph obtained from $G$ by identifying $W$ to a single new vertex
  $w$, and deleting any loop at $w$.  (As a typographical reminder, when depicting the vertex $w$ in our figures we will use a $\bullet$).  The following observation is
  an immediate consequence of this definition and Lemma~\ref{3edgecut}.

\begin{observation}
\label{bullet3ec}
If $\delta(Z)$ is a 3-edge-cut of $G^{\bullet}$ with $w \not\in Z$, then one of the following holds:
\begin{enumerate}
\item $X \subseteq Z$,
\item The graph induced by $Z$ contains at most one cycle.
\end{enumerate}
\end{observation}

Note that every ear of $G^{\bullet}$ is also an ear of $G$.  The three ears of $G^{\bullet}$ which are incident with the vertex $w$ will be called \emph{unusable} since we have little control over the 3-edge-cuts of $G$ which meet these three ears.  Every other ear of $G^{\bullet}$ will be called \emph{usable}.  We say that a cycle $C \subseteq G^{\bullet} \setminus \{w\}$ with $X\not\subseteq V(C)$ is an \emph{inner triangle} if $d( V(C) ) = 3$.  Note that in this case, Observation~\ref{bullet3ec} implies that $C$ has no chords. By Lemma~\ref{easydelconseq}, $C$ is a union of three distinct ears. Hence it contains exactly three vertices of degree 3, while all other vertices of $C$ have degree 2 (hence the name triangle).  Since $G$ is a subdivision of a 3-edge-connected graph, any two distinct inner triangles must be vertex disjoint.  We say that an ear $P$ of $G^{\bullet}$ is \emph{inner} if it is contained in an inner triangle and it is \emph{outer} otherwise. The following lemma follows from Lemma~\ref{3edgecut} applied to the graph obtained from $G^{\bullet}$ by suppressing the degree 2 vertices.

\begin{lemma}
  \label{removingedge}
Let $P$ be a usable ear of $G^{\bullet}$ and let $G' = \left( G - E(P) \right)^{\times}$.
\begin{enumerate}
\item If $P$ is inner, $G'$ is a subdivision of a 3-edge-connected graph.
\item If $P$ is outer, then the only 2-edge-cuts of $G'$ separating cycles have the form~$\delta(Y)$ where $Y \subseteq V(G^{\bullet} - w)$ contains a unique cycle $C$ which is an inner triangle in $G^{\bullet}$ (and contains an end of $P$).
\end{enumerate}
\end{lemma}

So, in short, removing an inner ear always results in a graph which is a subdivision of a 3-edge-connected graph, while removing an outer ear may result
in a graph with one or two cyclic 2-edge-cuts. If one of these cuts is of the form $\delta(V(C))$ for some cycle $C$ with $d(C)=3$ and $w \in V(C)$,
then $V(G^\bullet)\setminus V(C)$ induces a graph with exactly one cycle $C'$. Thus deleting any ear in $\delta(V(C))$ and then deleting one of the
two ears in $C'$ in the resulting graph produces a subdivision of Theta, which is a subdivision of a 3-edge-connected graph. All other cyclic edge
cuts produced by deleting an outer ear appear as $\delta(Y)$ where $Y$ contains a unique cycle $C$ which is an inner triangle (there are at most two
such cycles~$C$, one at each end of the removed outer ear~$P$).
We may further modify $G'$ by removing an ear of $G'$ contained in $C$ so as to return to a graph which is a subdivision of a 3-edge-connected graph.
This basic reduction technique will be exploited at great length throughout our proof.

A similar technique is based on the following lemma: we can remove three ears adjacent to the same vertex/inner triangle, and then
we remove up to three more ears, one from each inner triangle that now forms a 2-edge-cut. Again, this procedure yields a subdivision of a
3-edge-connected graph.

\begin{lemma}
  \label{removingvertex}
Let $P_1$, $P_2$, $P_3$ be usable ears of $G^{\bullet}$ such that
\begin{enumerate}
  \item either all of $P_1$, $P_2$, and $P_3$ are adjacent to the same vertex~$v \ne w$ of degree 3,
  \item or all of $P_1$, $P_2$, and $P_3$ are adjacent to the same inner triangle~$T$.
\end{enumerate}
Consider the graph~$G' = \left( G - \bigcup_{i=1}^3 E(P_i) \right)^{\times}$ (in the first variant)
or  $G' = \left( G - \bigcup_{i=1}^3 E(P_i) - E(T) \right)^{\times}$ (in the second variant).
Then the only 2-edge-cuts of $G'$ separating cycles have the form~$\delta(Y)$ where $Y \subseteq V(G^{\bullet} - w)$ contains a unique cycle $C_Y$,
which is an inner triangle in $G^{\bullet}$ (distinct from~$T$) and contains an end of some $P_i$ ($i=1$, $2$, or~$3$).
\end{lemma}

\begin{proof}
By arguments similar to those for Lemma~\ref{removingedge}, it is easy to see that deleting exactly one of $P_1,P_2$ and $P_3$, say $P_1$, and possible modification of its incident triangles yields a subdivision of a 3-edge-connected graph. Now $P_2\cup P_3$ in the vertex case or $P_2\cup P_3 \cup T'$ in the triangle case (where $T'\subseteq T$ is the leftover of $T$ after modification of $T$) forms a new ear. Delete the ear and possibly modify its incident triangles. Again, by arguments similar to those for Lemma~\ref{removingedge}, the resulting graph remains a subdivision of a 3-edge-connected graph.
\end{proof}

Going forward, we will ignore the original graph and instead focus our attention on $G^{\bullet}$ when searching for contractible or reducible subgraphs.  However, any contractible subgraph of $G^{\bullet}$ which contains only usable ears will also be a contractible subgraph of $G$.  Similarly, if $H$ is a reducible subgraph of $G^{\bullet}$ which contains only usable ears, then $H$ will also be a reducible subgraph of $G$ since the graph $(G - E(H))^{\times}$ will also be a subdivision of a 3-edge-connected graph.  So in short, as long as we avoid the unusable ears, we are free to operate in $G^{\bullet}$.

\section{Forbidden Configurations}

In this section we will establish a number of results which give us information about our minimal counterexample $G$.
In order to facilitate  calculations, we will introduce some graphical notation for working with possible subgraphs of $G$.  When depicting a subgraph
of~$G$ we will use an open circle $\circ$ to indicate a vertex with degree $\ge 3$ which is not $w$, and we will use a $\bullet$ to denote the vertex
$w$ (a vertex which might or might not be $w$ receives no symbol).  If there is an ear $P$ with ends $u,v$, then this will appear as a line segment
between $u$ and $v$ and we will sometimes mark this segment with a 0, 1, or 2 to indicate the length of $P$ mod 3 (we will not indicate vertices of
degree 2).  Finally, we will draw a dotted circle around an inner triangle and will draw a small dotted circle around a vertex to indicate that it is
not in an inner triangle (such vertex will be called a \emph{triad}).
We call these drawings \emph{configurations}, and we will say that a configuration is \emph{forbidden} if the corresponding
subgraph cannot appear in $G^{\bullet}$.

\begin{lemma}
\label{triangle_types}
Every inner triangle in $G^{\bullet}$ appears in a configuration of one of the following three types:

\begin{myfigure}[htp]
\centerline{\includegraphics[width=8cm]{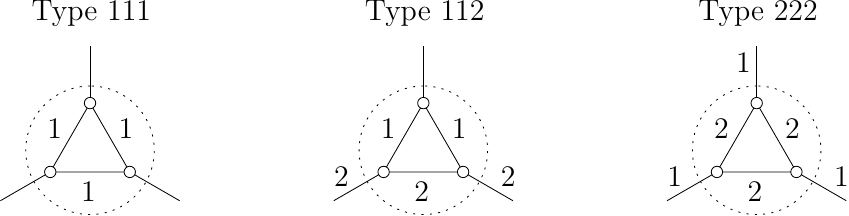}}
%\caption{Types of Inner Triangles}
\end{myfigure}

Furthermore, if $P$ is an ear with length 2 mod 3 which is contained in an inner triangle, then for every ear labelling $\psi$ of $P$ with gain at least 16, both newly formed ears of the $\psi$-removal of $P$ are equitable.
\end{lemma}

\begin{proof}  Consider the configurations appearing in the figure below.

\begin{myfigure}[htp]
\centerline{\includegraphics[width=8cm]{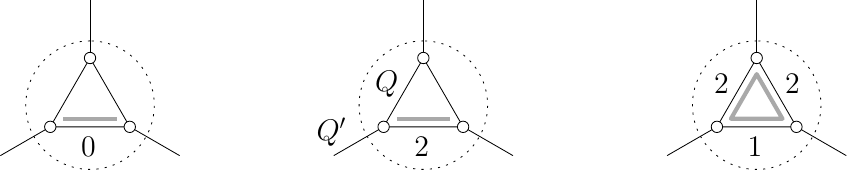}}
\end {myfigure}

In the rightmost configuration, the thick gray lines indicate a contractible subgraph by Lemma~\ref{easycont}, so this configuration is forbidden.  For the two configurations on the left in the figure, the thick gray lines indicate a partial ear decomposition consisting of just a single ear which we call $P$.

If $P$ has length 0 mod 3 (as on the left), then since $P$ is not equitable (Lemma~\ref{noequitable}), we may choose an ear labelling $\psi : E(P) \rightarrow \mathbb{Z}_3$ with gain at least 24 by Lemma~\ref{goodgain}.  The $\psi$-removal of $P$ is a weighted graph $G'$ which is still a subdivision of a 3-edge-connected graph.  Since there are just 5 ears of $G$ which are not ears of $G'$, we have $\bonus(G') \ge \bonus(G) - 20$ and we conclude that $P$ is reducible.  Thus, every inner ear of $G$ has nonzero length mod 3.

Next suppose that $P$ is an inner ear with length 2 mod 3 as in the middle and let $Q,Q'$ be the ears as shown in the figure.  By Lemma~\ref{goodgain} we may choose an ear labelling~$\psi$ of~$P$ with $\gain(\psi) \ge 16$ and we let $G'$~denote the $\psi$-removal of $P$.  Now $Q \cup Q'$ is an ear of $G'$, and if this ear is not equitable, then $\bonus(G') \ge \bonus(G) - 16$ and we have a reducible subgraph.  Therefore, $Q \cup Q'$ must be equitable in $G'$.  Since $Q$ cannot have length 0 mod 3 it must be that one of $Q,Q'$ has length 1 mod 3 and the other has length 2 mod 3.

Combining the above arguments we deduce that every inner triangle must have one of these three types, as desired.
\end{proof}

\begin{lemma}
\label{nozero}
No usable ear has length 0 mod 3.
\end{lemma}

\begin{proof}
Consider the configurations appearing in the figure below.

\begin{myfigure}[htp]
\centerline{\includegraphics[width=8cm]{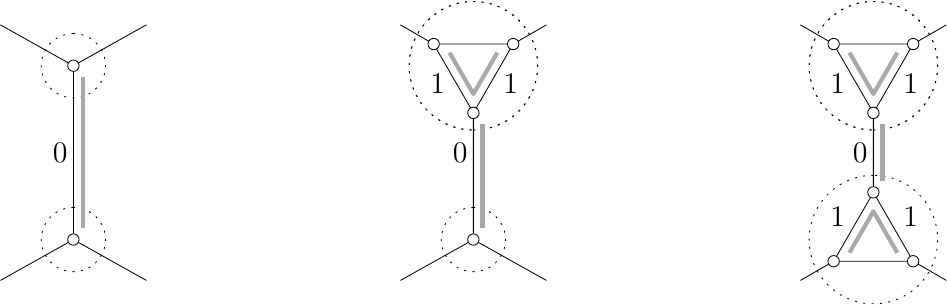}}
\end {myfigure}

For each of these configurations the figure indicates a partial ear decomposition.  Let $P$ be the outer ear with length 0 mod 3 as indicated and note
  that since $P$ is not equitable (by Lemma~\ref{noequitable}), we may choose an ear labelling~$\psi$ of~$P$ with gain at least 24
  (Lemma~\ref{goodgain}).   As usual, we will consider the $\psi$-removal graph $G'$.  If neither end of $P$ is in an inner triangle, then $G'$ is
  cyclically 3-edge-connected, so $P$ is reducible (clearly, $\bonus(G) - \bonus(G') \le 20$).
  Next suppose that some endpoint of $P$, call it $x$, is contained in an inner triangle.  Then by
  Lemma~\ref{triangle_types}, both of the ears of the inner triangle ending at $x$, call them $Q$ and $Q'$, must have length 1 modulo 3.  In the graph
  $G'$ it follows that $Q \cup Q'$ will be an ear with length $2$ modulo 3, so we may choose an ear labelling $\psi'$ with gain at least 16.  Now after
  forming the $\psi'$-removal we have repaired the connectivity problem caused by this inner triangle when we removed $P$.  Doing the same if necessary
  at the other end of $P$ (if it is also contained in an inner triangle) gives us a reducible subgraph.
\end{proof}

Our past two lemmas are rather easy reductions, but already give considerable structure to the graph $G^{\bullet}$.  These two lemmas already exhibit most of the basic ideas in our process, but we will need to consider rather more complicated configurations in going forward.

\subsection{Endpoint Patterns}

In the proof of our last lemma, after removing a certain outer ear, we needed to remove a little more to fix cyclic 2-edge-cuts caused by inner triangles.  In this section we will introduce some terminology to assist in doing the accounting in such circumstances.  So our setup for this section is as follows:  $P$ is an outer ear that we have decided to remove as shown in the figure below.

\begin{myfigure}[htp]
\centerline{\includegraphics[width=3.2cm]{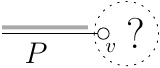}}
\end {myfigure}

We will focus our attention on just one end of $P$ which we call $v$.  If $v$ has degree at least four, then removing $P$ will not cause any ears at this end to merge.  If $v$ is a triad or in a cycle $C$ with $w$, then there are two other ears, say $Q,Q'$ incident with $v$ and removing $P$ will cause $Q$ and $Q'$ to merge into a single ear.  In this case we will say that $P$ has \emph{endpoint pattern} $Y_{0*}$, $Y_{11}$, $Y_{12}$, or $Y_{22}$ at $v$ depending on which of the configurations in the following figure it corresponds to.  Each of these patterns is associated with a cost which indicates the difference between the bonuses of $Q$ and $Q'$ in the original graph and the bonus of the ear $Q \cup Q'$ in the new graph obtained by removing $P$ (so this does not take into account the cost or gain from removing $P$).  Note that if $P$ has one of these endpoint patterns, then we don't need to worry about connectivity at this end after the removal of $P$.

\begin{myfigure}[htp]
\centerline{\includegraphics[width=13cm]{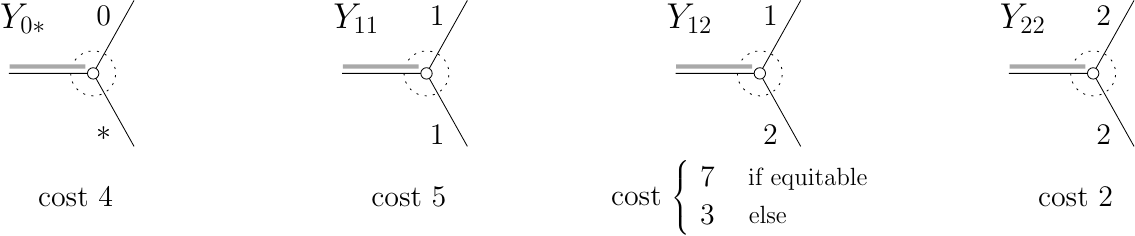}}
\end {myfigure}

So, for instance $Y_{11}$ has a cost of 5 since we have two ears of bonus 4 which merge to form an ear of bonus 3 in the new graph.  There is one special pattern here, $Y_{12}$ which has a variable cost depending on whether the newly formed ear (which will have length 0 mod 3) is equitable or not.

Next we consider the situation where we wish to remove an ear $P$ and we are interested in what happens at some end $v$ of $P$, but this vertex is contained in an inner triangle $C$.   In this case, when we remove $P$ to form the graph $G'$, this new graph will have a 2-edge-cut separating $C$ from the rest of the graph.  So, in order to maintain our desired connectivity, we will need to remove one of the two ears which comprise $C$ in the graph $G'$.  It will be convenient to absorb all of this in our notion of cost.  So, below we introduce four endpoint patterns $\Delta_{11}$, $\Delta_{11}'$, $\Delta_{12}$, and $\Delta_{22}$ with associated costs.
(The subscripts $a$ and $b$ in $\Delta_{ab}$ indicate that the two ears of the inner triangle which have $v$ as an end have lengths $a$ and $b$ modulo 3.)  In each case the cost gives an upper bound on the drop in bonus at this end minus the gain of the additional ear which is removed (when computing this drop in bonus, we count the five ears other than $P$ which have an end in the inner triangle).

\begin{myfigure}[htp]
\centerline{\includegraphics[width=13cm]{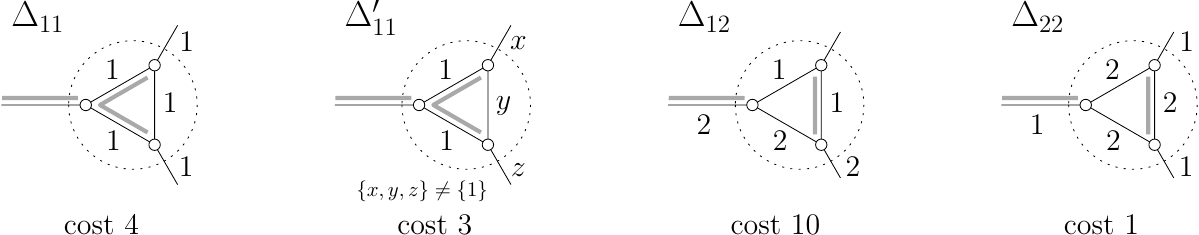}}
\end {myfigure}

So, for instance, the cost of $\Delta_{11}$ is $4$ because in moving from $G$ to $G'$ we lose a bonus of 4 for all 5 of the edges marked 1 for a total
of $-20$, we get a bonus of 0 or 4 in $G'$ for the newly formed ear since it has length 0 mod 3 (and might be equitable), and we get  $+16$ for the gain
of the ear which is removed. In endpoint pattern $\Delta_{11}'$, if either $x=y=z=2$ or the new ear obtained by deleting the two marked ears in the
inner triangle is inequitable, then the cost is $\le 1$; otherwise $y=1$ and $\{x,z\}=\{0,2\}$ and the total cost is 3. Note that the endpoint pattern
$\Delta_{12}$ cannot appear as the endpoint pattern of an ear with length 1 modulo 3 and similarly $\Delta_{22}$ cannot appear on an ear with length 2
modulo 3.  Next we will use this framework to establish some forbidden configurations.

\subsection{Forbidden Configurations}

\begin{lemma}
\label{forbid12}
The following configurations are forbidden.

\begin{myfigure}[htp]
\centerline{\includegraphics[height=3.5cm]{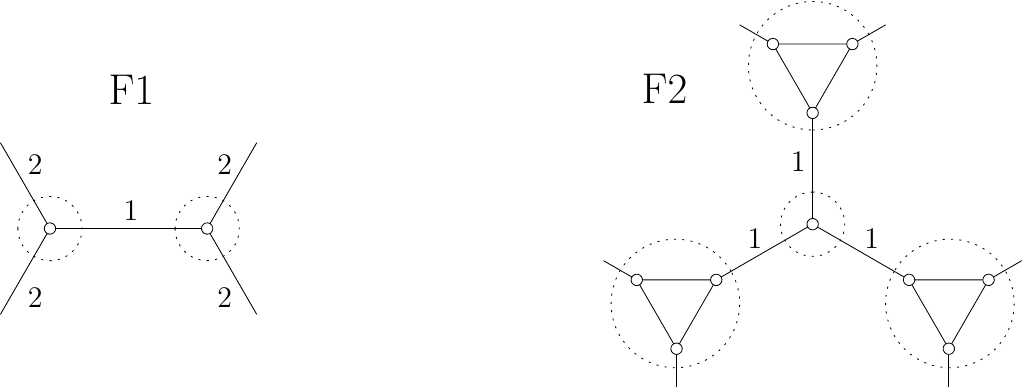}}
\end {myfigure}
\end{lemma}

\begin{proof}
First consider the configuration on the left, and let $P$ be the ear marked as having length 1 mod 3 in this figure and assume its ends are $u,v$.  The partial ear decomposition given by $P$ has gain 8, and the bonus of $P$ is 4.  Since the endpoint patterns associated with removing $P$ at $u$ and $v$ are both of cost 2, this gives a reducible subgraph.  Although it is obvious in this case, let us note that this endpoint cost calculation would not be valid if in the graph obtained by removing $P$, both $u$ and $v$ were contained in a single ear.  In general we will need to be careful to ensure that these endpoint cost calculations are independent.

Next consider the configuration on the right, let $P_1, P_2, P_3$ be the ears marked as having length 1 mod 3 in this figure, and let $v$ be the common
  endpoint of $P_1, P_2, P_3$.  Now $P_2 \cup P_3, P_1$ is a partial ear decomposition with gain 24 and the bonus of $P_1 \cup P_2 \cup P_3$ is 12.
  Since the endpoint costs associated with removing each $P_i$ (at the endpoint other than $v$) are all at most 4, this gives a reducible configuration.
  Again, there is a potential danger here that there is an outer ear with an end in two of the inner triangles from this configuration (as this would
  invalidate our cost calculation), but this would result in a 3-edge-cut of $G^{\bullet}$ of a type that contradicts Observation~\ref{bullet3ec}, so it
  does not happen.
  Also note that the removal of~$P_1 \cup P_2 \cup P_3$ is possible according to Lemma~\ref{removingvertex}.
\end{proof}

In preparation for our next forbidden configurations, we will prove a couple of lemmas which will help us to control ears of length 1 mod 3 and allow us to arrange for the cost of endpoint pattern $Y_{12}$ to be 3 (instead of 7) in certain cases.

\begin{lemma}\label{le:length1}
If $P$ is a usable ear of $G^{\bullet}$ with length 1 mod 3, then $|E(P)| = 1$.
\end{lemma}

\begin{proof}
  Suppose (for a contradiction) that $P$ has length 1 mod 3 and $|E(P)| \neq 1$.  By Lemma~\ref{noequitable} there is an ear labelling
  $\psi$ of $P$ with $\supp(\phi) = E(P)$.  So $\gain(\psi) = 24 | \supp(\phi)| - 16 |E(P)| = 8 |E(P)| \ge 32$.  If $P$ is an inner ear then it is a
  reducible subgraph.  Otherwise, $P$ is an outer ear with bonus 4, and  the endpoint patterns at the ends of $P$ each have cost at most 7, so $P$ may
  be extended to a reducible subgraph.
\end{proof}

\begin{lemma}\label{le:Y12}
\label{clever_choice}
Let $H$ be a weighted 2-edge-connected graph, let $v \in V(H)$ have $\deg(v) = 3$ and assume that $P, Q, Q'$ are distinct ears of $H$ with endpoint $v$.
If $|E(P)| = 1$ and $Q$ is not equitable, then there exists an ear labelling $\psi$ of $P$ with gain 8 so that in the $\psi$-removal of $P$, the ear
$Q \cup Q'$ is not equitable.
\end{lemma}

\begin{proof}
Let $\psi_1, \psi_2, \psi_3$ be the ear labellings of $P$, and suppose that $\psi_3$ is the zero function (so $\psi_1, \psi_2$ have  support $E(P)$).
Let the ear labellings of $Q$ ($Q'$) be denoted $\sigma_i$ ($\sigma_i'$) for $1 \le i \le 3$.  When we perform a $\psi_i$-removal of $P$, the resulting
weighted graph will have $Q \cup Q'$ as an ear, and each ear labelling of this ear will have the form $\sigma_j \cup \sigma_k'$ for some $1 \le j,k \le 3$.
It follows from basic principles that (working with our indices modulo 3) we may assume that $\sigma_j \cup \sigma'_{j+i}$ is an ear labelling of
$Q \cup Q'$ in the $\psi_i$-removal of $P$.

If $|E(Q \cup Q')|$ is not divisible by~3, then $Q \cup Q'$ is inequitable. Thus we may assume that
either $|E(Q)|$ is $1 \bmod 3$ and $|E(Q')|$ is $2 \bmod 3$
or both $|E(Q)|$, $|E(Q')|$ are divisible by~3.
  By a straightforward case-analysis (and using the fact that $Q$ is inequitable),
it follows that either the $\psi_1$- or $\psi_2$-removal of $P$ will have $Q \cup Q'$ as an inequitable ear, thus proving the result.
\end{proof}

\begin{lemma}
\label{forbiddentriads}
The following configurations are forbidden in $G^{\bullet}$

\begin{myfigure}[htp]
\centerline{\includegraphics[width=6cm]{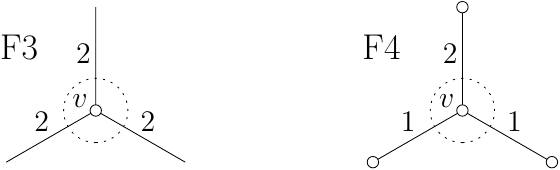}}
\end {myfigure}

\end{lemma}

\begin{proof} Suppose (for a contradiction) that we have the configuration on the left in the statement of the lemma, and choose an ear $P$ which has $v$ as
  an end, but does not have the special vertex~$w$ as an end.  It now follows from a straightforward calculation that $P$ may be extended to a reducible
  subgraph.  Namely, the gain of the partial ear decomposition $P$ is 16, the bonus of $P$ is 3, and the endpoint costs at its ends will be 2 at $v$ and
  at most 10 at the other end. (We may need to delete another ear as suggested by the endpoint costs -- that is why we may need to extend~$P$
  to a reducible subgraph.)

\begin{myfigure}[htp]
\centerline{\includegraphics[width=8cm]{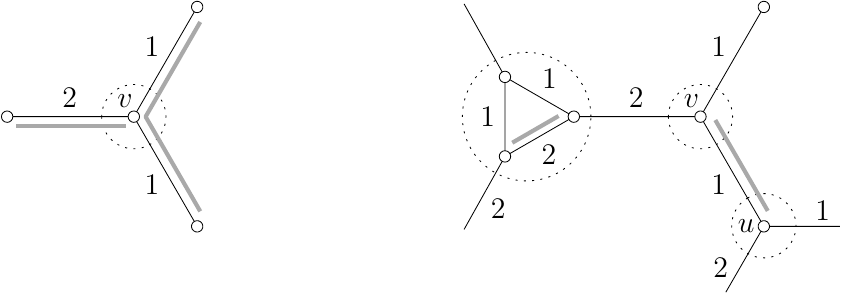}}
\end {myfigure}

Next suppose (for a contradiction) that we have the  configuration on the right in the statement of the lemma, and consider the partial ear decomposition
  with gain 32 indicated on the left in the above figure.  We may assume that this does not extend to a reducible subgraph, and this implies that the sum
  of the endpoint costs of these ears (at the ends other than $v$) must be at least 22.  Again, these cost calculations are independent as otherwise we would have a violation to Observation~\ref{bullet3ec} (or to the assumption that $v$ is not in an inner triangle).  The only way this is possible is for one of these endpoint patterns to be $\Delta_{12}$ and for another to be $Y_{12}$.  So, we may now assume that we have the configuration shown on the right in the above figure.  We have indicated a partial ear decomposition for this case with gain 24.  Let this partial ear decomposition be $P_1, P_2$ where $P_1$ is the ear with ends $u,v$.  By Lemma~\ref{le:length1}, $|E(P_1)|=1$. Now by applying Lemma~\ref{le:Y12}, we may choose an ear labelling $\psi_1$ of $P_1$ which has gain 8 (as required) and has the additional property that in the $\psi_1$-removal graph $G'$, the ear which contains the vertex $u$ will not be equitable.  Next we choose an ear labelling $\psi_2 : E(P_2) \rightarrow \mathbb{Z}_3$ with gain $\ge 16$ and let $G''$ be the $\psi_2$-removal from $G'$.  The ear of $G''$ containing $v$ has length 1 mod 3, and thus $\bonus(G'') = \bonus(G) - 32 + 8$ giving us a reducible subgraph.
\end{proof}

\begin{lemma}\label{le:subcubic}
$G^{\bullet}$ is subcubic (i.e., it has maximum degree 3).
\end{lemma}

\begin{proof}
First let us suppose (for a contradiction) that there is a pair of usable ears $P,P' \in E(G)$ with the same ends, say $u,v$. Neither $u$ nor $v$ has degree 3 by Lemma~\ref{easydelconseq}. However, now $P$ is a reducible subgraph since this partial ear decomposition has gain at least 8, and the bonus of $P$ is at most 4.
%By Lemmas~\ref{nozero} and~\ref{easycont} we may assume (without loss) that $P$ has length 2 mod 3.    However, now $P$ is a reducible subgraph since this partial ear decomposition has gain 16, and the bonus of $P$ is just 3.

So, we may now assume that no such ears exist.  Next we suppose (again for a contradiction) that there is a vertex $u$ with degree $\ge 4$ and note that $u \neq w$.  Since there are exactly 3 ears with $w$ as an end and at least 4 with $u$ as an end, there must be an ear $P$ with one end $u$ and another end $v$ which is not $w$ and not a neighbour of $w$.

If $v$ has degree at least 4, then $P$ is a reducible subgraph since the partial ear decomposition $P$ has gain at least 8, and the bonus of the resulting graph differs from the original only by the bonus of $P$, which is at most 4.  So we may assume that $v$ has degree 3.  If $P$ has length 2 mod 3, then again $P$ is a reducible configuration since it has gain 16 and the endpoint cost of $P$ at $v$ is at most 10.  So, we may assume that $P$ has length 1 mod 3.

If $v$ is in an inner triangle, then the endpoint cost of $P$ at $v$ is at most 4, and again $P$ extends to a removable subgraph.  So, we may assume that $v$ is a $Y_{ab}$ configuration.  The endpoint pattern of $P$ at $v$ cannot be $Y_{12}$ by Lemma~\ref{forbiddentriads},
and if it is either $Y_{*0}$ or $Y_{22}$, then the associated cost is at most 4, and again we find that $P$ is a reducible subgraph.  The only remaining case is when this endpoint pattern is $Y_{11}$ as shown in the figure below.

\begin{myfigure}[htp]
\centerline{\includegraphics[width=4cm]{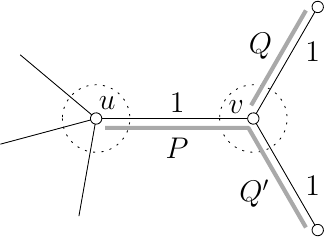}}
\end {myfigure}

The partial ear decomposition $P \cup Q', Q$ as shown has gain 24, and $\bonus(P \cup Q \cup Q') = 12$ so $P \cup Q \cup Q'$ will extend to a reducible
  configuration if the endpoint patterns of $Q$ and $Q'$ (at the vertices other than $v$) have cost at most 12.  It follows that we have a reducible subgraph unless the
  endpoint patterns of $Q$ and $Q'$ at the vertex other than $v$ are both $Y_{12}$.  However, in this case Lemma~\ref{clever_choice} permits us to obtain an endpoint cost
  of 3 when removing the ear~$Q$ (i.e., we may choose an ear labelling $\psi$ of $Q$ with gain 8 so that in the $\psi$-removal graph $G'$ the ear containing an end of $Q$ other than $v$ is not equitable).  Using this gives us a removable subgraph, thus contradicting the assumption that $G^{\bullet}$ has a vertex of degree $>3$.
\end{proof}

\section{Taming Triangles}

In this section we will prove some lemmas which will help to tame the possible behaviour of the inner triangles in  $G^{\bullet}$.

\begin{lemma}
\label{outer2force}
Let $P$ be a usable outer ear with length 2 modulo 3.  Then either both ends of $P$ are triads with pattern $Y_{12}$ or $P$ has one end which is a
triad, and the other has pattern $\Delta_{12}$.
\begin{figure}[htp]
\centerline{\includegraphics[width=11cm]{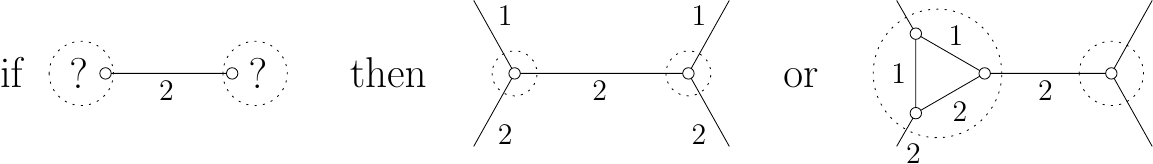}}
\end {figure}
Furthermore, if $P$ is a usable outer ear with an endpoint $u$ of pattern $Y_{12}$, then every ear labelling $\psi$ of $P$ with $\gain(\psi) \ge 16$ has the property that the ear of the $\psi$-removal containing $u$ is equitable.

\end{lemma}

\begin{proof}
The proof of this lemma calls on two more forbidden configurations as shown in the figure below.

\begin{myfigure}[htp]
\centerline{\includegraphics[width=8cm]{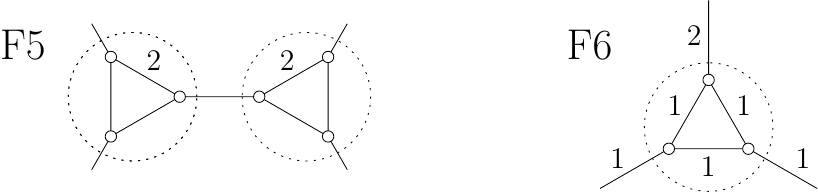}}
\end {myfigure}

To prove that F5 is forbidden, consider the partial ear decomposition indicated on the left in the figure below.  This decomposition has gain 32 and the
  ear of the resulting weighted graph containing $Q$ will have nonzero length mod~3 (Lemma~\ref{triangle_types}).
  It follows that this yields a reducible subgraph.

\begin{myfigure}[htp]
\centerline{\includegraphics[width=9cm]{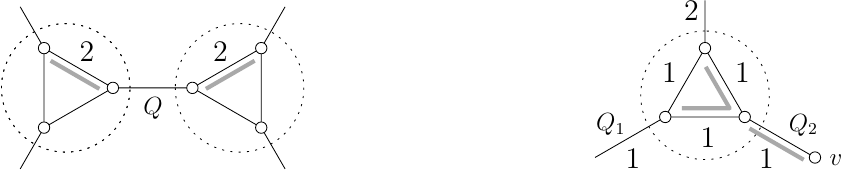}}
\end{myfigure}

To prove that F6 is forbidden, consider the ears $Q_1$ and $Q_2$ as shown on the right in the above figure.  By our connectivity, it is impossible for $Q_1$ and $Q_2$ to both have $w$ as an endpoint, so we have assumed (without loss) that $Q_2$ has an endpoint $v \neq w$.  Consider the partial ear decomposition indicated in this figure.  This decomposition has gain 24 and we can see that there is an ear with length 1 mod 3 formed by removing these ears.  It follows that we have a reducible subgraph unless the endpoint cost associated with $Q_2$ at $v$ is greater than 5.  The only possibility here would be for this endpoint pattern to be of type $Y_{12}$, but in this case Lemma~\ref{clever_choice} still guarantees us a removable subgraph.  It follows that this configuration is forbidden, as claimed.

Equipped with these forbidden configurations, the proof of the lemma is straightforward.  A simple check of costs of endpoint patterns reveals that either we have the first outcome or $P$ has at least one endpoint pattern of type $\Delta_{12}$.   Since $P$ cannot have two endpoint patterns of type $\Delta_{12}$  by F5, we may assume it has exactly one.  If the other end of $P$ is a triad, then we have nothing left to prove.  Otherwise it is contained in an inner triangle and the associated endpoint pattern must have cost at least 4.  However, this gives us the forbidden configuration F6.  The additional claim is straightforward.
\end{proof}

\begin{lemma}
\label{222-op-triangle}
The following configuration is forbidden.

\begin{myfigure}[htp]
\centerline{\includegraphics[width=4cm]{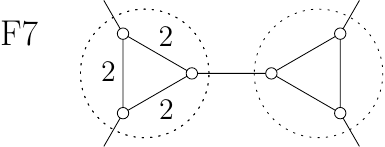}}
\end {myfigure}
\end{lemma}

\begin{proof} Suppose (for a contradiction) that this configuration is present in $G^{\bullet}$ and let $R$ denote the ear between the two inner triangles in this configuration.  Note that $R$ must have length 1 mod 3 by Lemma~\ref{triangle_types}.  The endpoint costs associated with the two ends of $R$ must be greater than 4 as otherwise $R$ may be extended to a reducible subgraph.  It follows from that  $R$ has to have a $\Delta_{11}$ endpoint pattern,  so we have the configuration in the figure below.  Here it is impossible for both $P$ and $P'$ to have $w$ as an endpoint, so we have assumed (without loss) that $P$ has an end $v \neq w$.

\begin{myfigure}[htp]
\centerline{\includegraphics[width=4cm]{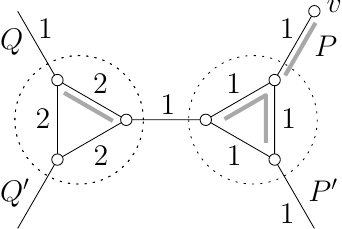}}
\end {myfigure}

It is not possible (by connectivity) for both of the ears $Q$ and $Q'$ to have $v$ as an endpoint.   Similarly, it is not possible for there to be an
  inner triangle $C$ which contains $v$ and an endpoint of both $Q$ and $Q'$.  Accordingly, we may assume (without loss) that $Q$ does not have $v$ as
  an end and there is no inner triangle containing $v$ and an end of $Q$.  Now consider the partial ear decomposition indicated in the figure, and note
  that it has gain 40.  If the cost of the endpoint pattern of $P$ at $v$ is at most 6, a straightforward calculation shows that this
  can be extended to a reducible
  subgraph (here we use the fact that after removing these ears, the ear containing $P'$ will have length 2 mod 3).  Otherwise, this endpoint pattern
  must be of type $Y_{12}$, but in this case we can use Lemma~\ref{clever_choice} to provide a reducible subgraph, a contradiction.
\end{proof}

%\subsection*{The Graph $\mathbf{G^{\Delta}}$}
\subsection[The Graph $G^{\Delta}$]{The Graph $\mathbf{G^{\Delta}}$}

We let $G^{\Delta}$ denote the (unweighted) graph obtained from $G^{\bullet}$ by identifying each inner triangle to a (distinct) vertex and then suppressing all degree 2 vertices.  
So every edge $e$ of $G^{\Delta}$ \emph{corresponds} to an outer ear $P$ of $G^{\bullet}$ and we say that $e$ has \emph{residue} 0, 1, or 2 if this is the length of $P$ mod 3.  
If $v$ is a vertex of $G^{\Delta}$ which was formed by identifying an inner triangle $T$, then we say that $v$ \emph{corresponds} to $T$, and more generally, 
if $H$ is a subgraph of $G^{\Delta}$, then we say that $H$ \emph{corresponds} to the subgraph of $G^{\bullet}$ consisting of the outer ears of $G^{\bullet}$ 
corresponding to the edges of $H$ together with all inner triangles of $G^{\bullet}$ corresponding to vertices of $H$.  
Working with $G^{\Delta}$ will prove to be convenient as evidenced by the next couple of lemmas.  
Recall that a graph is \emph{cyclically $k$-edge-connected} if every edge cut separating cycles has size at least $k$ and 
note that cyclically $4$-edge-connected \emph{cubic} graph must be 3-edge-connected. 

\begin{lemma}
\label{cyc4ec}
$G^{\Delta}$ is cubic and cyclically 4-edge-connected.
\end{lemma}

\begin{proof}
Note that by the connectivity of $G^{\bullet}$ the graph $G^{\Delta}$ must be 3-edge-connected.
Moreover, $G^\Delta$ is cubic, since $G^\bullet$ is subcubic (Lemma~\ref{le:subcubic}).
Suppose for a contradiction that $G^{\Delta}$ has an edge-cut $\delta(Z)$ of size at most three which separates cycles, and assume (without loss) that $w \not\in Z$.  If $Z$ induces a graph with at least two cycles in $G^{\Delta}$, then
we have a contradiction to the choice of cut used in constructing $G^{\bullet}$ (Observation~\ref{bullet3ec}). 
Otherwise, $Z$ contains a unique cycle $C$. If some vertex in $Z$ corresponds to an inner triangle in $G^{\bullet}$, then again we have a contradiction to the choice of cut in constructing $G^{\bullet}$.  Otherwise the subgraph of $G^{\Delta}$ induced by $Z$ will correspond to a cycle $C$ in $G^{\bullet}$ satisfying $d(C) = 3$, but then $C$ would be an inner triangle of $G^{\bullet}$ and this yields a contradiction.
\end{proof}

\begin{corollary}
\label{cyc4eccor} Let $W\subseteq V(G^{\Delta})$. 
\begin{itemize}
\item If $|d(W)| = 3$, then $W$ or $\overline{W}=V(G)^{\Delta}\setminus W$ consists of a single vertex.
\item If $|d(W)| = 4$ and $W$ does not contain a cycle, then  $W$~induces a single edge.
\end{itemize}
\end{corollary}

\begin{proof} Suppose $|d(W)|\in\{3,4\}$. If $|d(W)|=3$ then the edge cut doesn't separate cycles since $G^{\Delta}$ is  cyclically 4-edge-connected; suppose this is also the case when $|d(W)|=4$. Since $G$ is 3-edge-connected this means $W$ or $\overline{W}$ induces a tree in $G^{\Delta}$. Since $G^{\Delta}$ is cubic, and trees with at least two vertices have at least two leaves, there are only two possibilities: $|d(W)|=3$ and this tree is a single vertex, or  $|d(W)|=4$, and this tree is a single edge.
\end{proof}

\begin{lemma}
\label{notk4}
$G^{\Delta}$ has at least six vertices and girth $\ge 4$.
\end{lemma}

\begin{proof} By the previous lemma, it suffices to show that $|V(G^{\Delta})| \ge 6$.
The graph~$G^{\Delta}$ cannot have two vertices by the definition of $G^{\bullet}$ (which must have at least two cycles not containing $w$).  So, if
  the lemma fails, then by the previous lemma we must have $G^{\Delta} \cong K_4$.  In this case at least one vertex of $G^{\Delta}$  must
  correspond to an inner triangle, as otherwise we would have a contradiction to the definition of $G^{\bullet}$.  So, $G^{\bullet}$ must contain either
  1, 2, or 3 inner triangles.  If it has at least two, then Lemma~\ref{222-op-triangle} implies that there is no triangle of type 222.  By repeatedly
  applying Lemma~\ref{outer2force} we may conclude that $G^{\bullet}$ appears as one of the graphs in the following figure.

\begin{myfigure}[htp]
\centerline{\includegraphics[width=13cm]{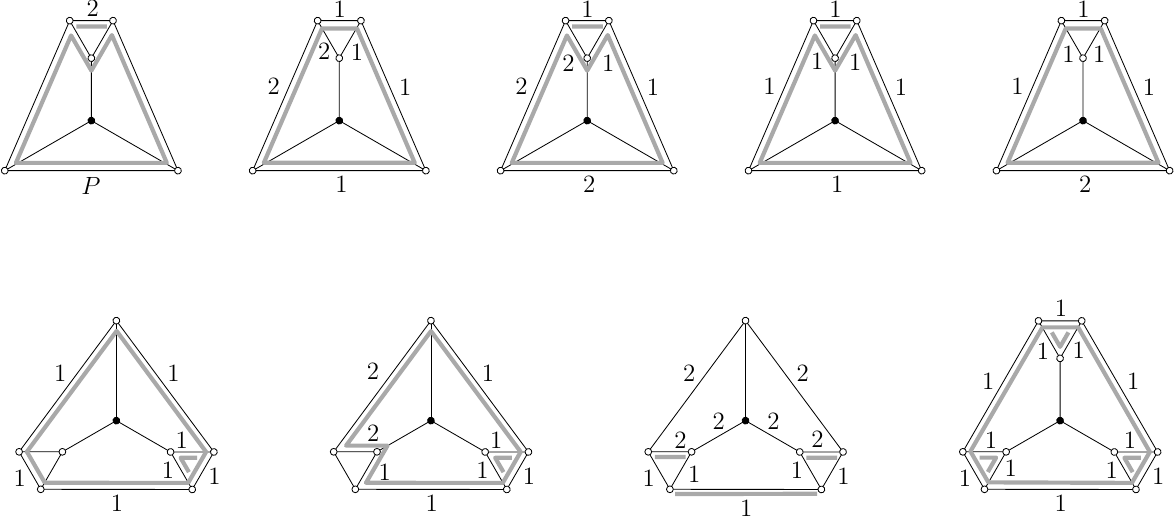}}
\end {myfigure}

This figure also indicates either a contractible or reducible subgraph of $G^{\bullet}$ in each case.
In the penultimate configuration in our list we indicated a reducible subgraph of gain 40, while the ears involved have total bonus of 38.
In each of the other eight configurations  we have indicated a contractible subgraph~$H$ and an ear decomposition of $H$ which has gain at least as large as the total bonus on the ears of $H$.  The only tricky case here is the first configuration.  In this case it follows from our triangle types that the indicated gray cycle containing the ear $P$ will have the same length modulo 3 as $P$.  In particular, this is nonzero, so the indicated subgraph $H$ has an ear decomposition of gain at least 24.
%The penultimate configuration in our list instead has a reducible subgraph.  The depicted partial ear decomposition has gain 40.
%The total bonus on all of the ears is $38$ plus the bonus of the ear $Q$.  However, it follows from Lemma~\ref{outer2force} that in the weighted graph we obtain by removing these ears (to achieve the desired bonus) $Q$ will merge with an equitable path to form a new ear with the same bonus as $Q$.  It follows that this is indeed a reducible subgraph.
\end{proof}

\begin{lemma}
\label{even2}
If $v \in V(G^{\Delta})$ is not adjacent to $w$, then there are an even number of edges incident with $v$ which have residue 2.
\end{lemma}

\begin{proof}
If the vertex $v$ is a triad in $G^{\bullet}$, then this follows from Lemmas~\ref{nozero} and~\ref{forbiddentriads}.  Otherwise, $v$ corresponds to a
  inner triangle $C$ in $G^{\bullet}$.  The ears of $G^{\bullet}$ associated with the edges $\delta_{G^{\Delta}}(v)$ are all usable and outer, so
  Lemma~\ref{outer2force} implies that if one of these ears has length 2 mod 3, then $C$ is a $112$ triangle and there are exactly two such ears.
  (If there are three, we use Lemma~\ref{outer2force} again to get a contradiction.)
\end{proof}

\subsection{Type 222 and 112 Triangles}

Our next goal will be to prove that  $G^{\bullet}$ does not have a triangle of type 222 and we will prove a lemma showing that it has at most one 112 triangle, and giving some further structure if one exists.  We proceed next with a couple more forbidden configurations.

\begin{lemma}
\label{2near112}
The following configurations are forbidden.

\begin{myfigure}[htp]
\centerline{\includegraphics[width=8.5cm]{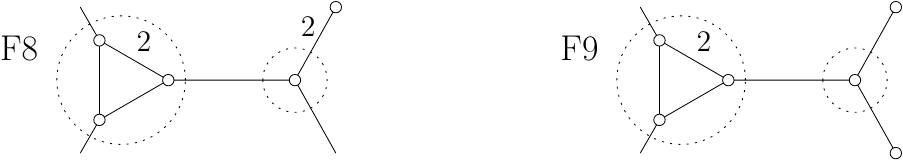}}
\end {myfigure}
\end{lemma}

\begin{proof}
The figure below indicates the reductions we will use.

\begin{myfigure}[htp]
\centerline{\includegraphics[width=8.5cm]{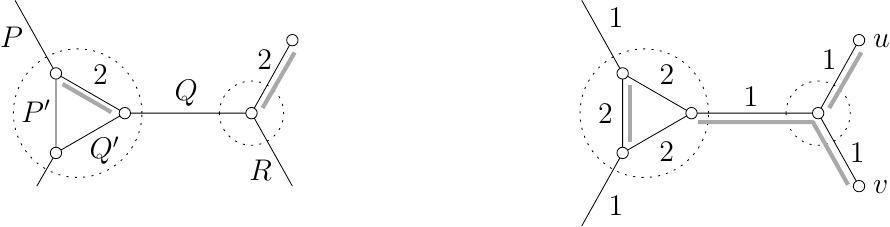}}
\end {myfigure}

For the configuration on the left, we have indicated a partial ear decomposition consisting of two disjoint ears, each of which has an ear labelling with gain 16 (for a total gain of 32).  However, it follows from Lemma~\ref{triangle_types} that in the weighted graph obtained by removing these ears (to achieve gain $\ge 32$), both $P \cup P'$ and $Q \cup Q'$ will be equitable, and thus the bonus of the ear containing $R$ in this new weighted graph is the same as the bonus of $R$ in the original.  Therefore, the drop in bonuses on the pictured ears minus the bonuses on the (pictured) newly created ears is at most 22.  Since no endpoint pattern has a cost greater than 10, this gives us a reducible configuration, as desired.

For the configuration on the right in the statement of the lemma, we may assume by the above argument and Lemma~\ref{forbiddentriads} that all three ears incident with the pictured triad have length 1 mod 3.  It then follows from our triangle types that this configuration must have the lengths as indicated in the above figure.  Here we have indicated a partial ear decomposition with gain 40.  The total bonus of the pictured ears is 29, so we will have a reducible configuration unless the sum of the costs of the two endpoint patterns at $u$ and $v$ is at least~12.  The only way this is possible is for at least one of $u,v$ to have endpoint pattern $Y_{12}$, so we may assume this happens at $u$.  However, in this case we may apply Lemma~\ref{clever_choice} to arrange for the cost of the endpoint pattern at $u$ to be 3, thus giving us a reducible configuration.
\end{proof}

\begin{lemma}\label{le:no222}
There is no inner triangle of type $222$.
\end{lemma}

\begin{proof} We will argue first in the graph $G^{\Delta}$.  Let $x$ be a vertex of $G^{\Delta}$ corresponding to a $222$ triangle.  It follows from Lemma~\ref{222-op-triangle} that no neighbour of $x$ corresponds to an inner triangle.  It then follows from the previous lemma that every neighbour of $x$ must either be equal to $w$ or adjacent to $w$.  Since $G^{\Delta}$ has girth $\ge 4$ (by Lemma~\ref{notk4}) it must be that every neighbour of $x$ is adjacent to (but not equal to) $w$.   Now consider the subset of $V(G^{\Delta})$ consisting of $w$, $x$, and the neighbours of $x$.  There are just three edges with exactly one endpoint in this set.  So, by Lemma~\ref{cyc4ec} we find that there is just one vertex $z$ of $G^{\Delta}$ not in this set, and thus $G^{\Delta} \cong K_{3,3}$.  It follows from the previous lemma that all three edges of $G^{\Delta}$ incident with $z$ have residue 1.  Consequently the subgraph of $G^{\bullet}$ consisting of all of the usable ears must be one of the following.

\begin{myfigure}[htp]
\centerline{\includegraphics[height=3cm]{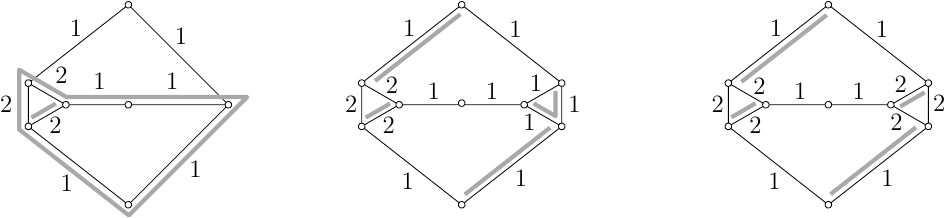}}
\end {myfigure}

In the leftmost graph we have indicated a contractible subgraph $H$.  To see this, note that the indicated ear decomposition of $H$ has gain 32, but the
  total bonus on the ears of $H$ is 25.  For the other two graphs we have indicated partial ear decompositions both of which have a gain of 48.  This
  gain is at least the sum of the bonuses on the ears which do not have $w$ as an endpoint (i.e., the ears which are fully pictured).
By Lemma~\ref{triangle_types}, all the ears created by deleting an inner ear of length 2 mod 3 are equitable. Thus, in the third configuration, if $P$
  is an ear of $G^{\bullet}$ which is incident with $w$, then either $P$ will still be an ear after the indicated reduction, or $P$ will merge with an
  equitable path to form a larger ear $P'$ with the same bonus as $P$. This shows that the third configuration is indeed reducible.
Now consider the middle configuration. Let $Q$~denote the marked ear incident with the type 111 triangle and $P$, $P'$ the two marked ears contained in
  the 111 triangle. Also, let $v$~be the topmost neighbour of $w$.
A similar reasoning as for the third configuration shows that the middle configuration is reducible, unless the ear between $v$ and $w$ in $G^\bullet$
  is inequitable of length 0 mod 3, and the ear created by deleting $P$ and $P'$ is inequitable (and these two ears form an equitable ear in the
  resulting graph).
Thus we may assume that this is the case and, by symmetry, all ears incident with $w$ have length 0 mod 3. In this case the subgraph given by $P, P'$
  and $Q$ is reducible.
\end{proof}

\begin{lemma}
\label{tame112}
There is at most one type 112 triangle in $G^{\bullet}$.  Furthermore, if it exists, then it appears in a configuration of the following form.

\begin{myfigure}[htp]
\centerline{\includegraphics[width=3cm]{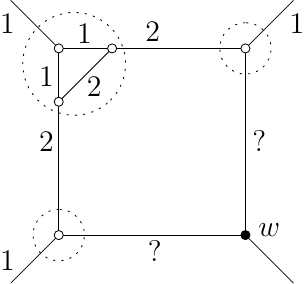}}
\end {myfigure}
\end{lemma}

\begin{proof}
Let $C$ be a 112 triangle in $G$ and let $Q$ be the unique ear of~$G$ with length 2 modulo 3 which is contained in $C$.
Let $P_1,P_2$ be the two ears which are not contained in $C$ but have an endpoint in common with $Q$ (so $P_1$ and $P_2$ both have length 2 modulo 3).  For $i=1,2$ let $v_i$ be the end of $P_i$ which is not contained in $C$.  It follows from Lemma~\ref{outer2force} that $v_i$ is not contained in an inner triangle for $i=1,2$.  It then follows from Lemma~\ref{2near112} that $v_i$ must either be equal to $w$ or joined to $w$ by a single ear for $i=1,2$.  If either $v_i$ is equal to $w$, then this results in a cycle of length at most three in $G^{\Delta}$ which is a contradiction.  Therefore, both $v_1$ and $v_2$ are joined to $w$ by a single ear.  Now, the outward pointing ears in the configuration from the statement which are marked 1 must indeed have length 1 mod 3 by applications of Lemma~\ref{outer2force} and Lemma~\ref{2near112}.  It follows immediately from this that it is impossible for $G^{\bullet}$ to contain another triangle of type 112.
\end{proof}

\section{Cycles and Edge-Cuts of Size Four}

In this section we will be considering 4-edge-cuts in $G^{\Delta}$ and in particular 4-cycles in $G^{\Delta}$ in order to find more reductions and complete the proof.

\begin{lemma}
\label{all1insq}
If $C$ is a 4-cycle in $G^{\Delta}$ not containing $w$, and every edge of $C$ has residue 1, then the corresponding subgraph of $G^{\bullet}$ is one of the following.

\begin{myfigure}[htp]
\centerline{\includegraphics[width=12cm]{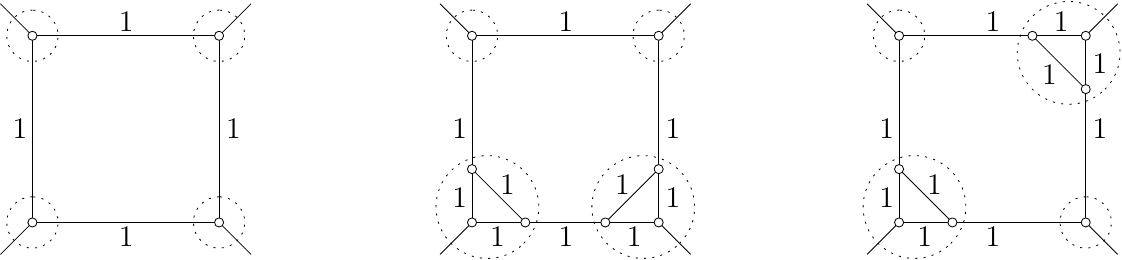}}
\end {myfigure}
\end{lemma}

\begin{proof} This proof will call upon the following configurations.

\begin{myfigure}[htp]
\centerline{\includegraphics[width=12cm]{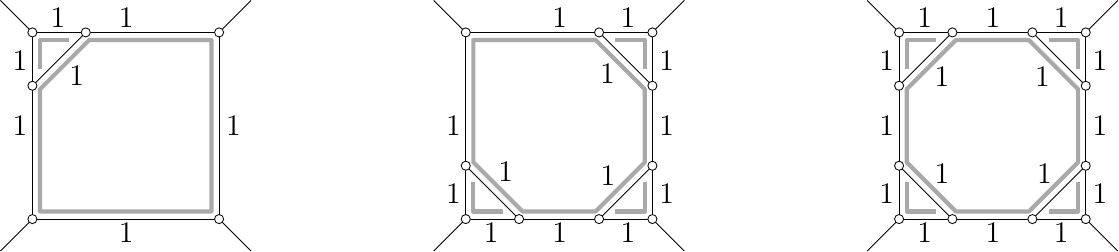}}
\end {myfigure}

Each vertex of $C$ corresponds to either a triad or an inner triangle of type 111 in $G^{\bullet}$. By Lemma~\ref{le:no222}, all these triangles are of type 111. Up to symmetry, this gives six possibilities; three of them are excluded by the contractible subgraphs indicated in the above figure, while the other three are those from the statement of the lemma.
\end{proof}

\begin{lemma}
Let $C \subset G^{\Delta}$ be a 4-cycle and assume $C$~does not contain the vertex~$w$. %and $C$~does not contain a vertex corresponding to $112$ triangle.
Then the corresponding subgraph of $G^{\bullet}$ cannot be as follows:
%There is no 4-cycle $C \subseteq G^{\Delta}$ without the vertex $w$ or a vertex corresponding to $112$ triangle such that the corresponding subgraph
%  of $G^{\bullet}$ is as follows:

\begin{myfigure}[htp]
\centerline{\includegraphics[height=3cm]{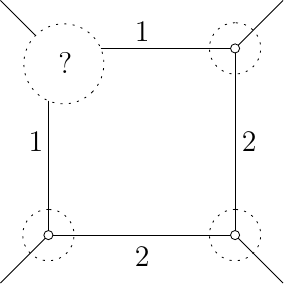}}
\end {myfigure}

Here, the question mark in the top right corner denotes either a triad or an inner triangle.
\end{lemma}

\begin{proof} Note that by Lemma~\ref{outer2force}, all three vertices in the above figure must be incident with two ears of length 2 mod 3 and one ear of length 1 mod 3. By Lemma~\ref{tame112} the top left vertex of $C$ does not correspond to a type 112 triangle. By our assumptions, the configuration in the statement of the lemma must extend to that appearing on the left or the middle in the figure below.

\begin{myfigure}[htp]
\centerline{\includegraphics[height=2.75cm]{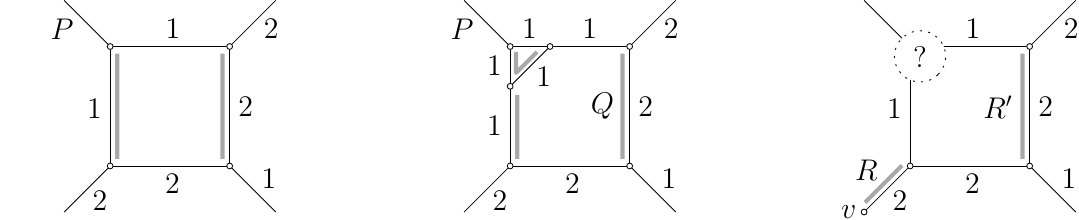}}
\end{myfigure}

Assume for starters that $G^{\Delta} \cong K_{3,3}$.  In this case we claim that the configurations on the left and middle above indicate reducible subgraphs (here the required connectivity follows from the assumption $G^{\Delta} \cong K_{3,3}$).
In the leftmost figure, the indicated partial ear decomposition has gain 24 and by Lemma~\ref{outer2force}, the weighted graph obtained by removing these ears (to achieve a gain of 24) will have an ear containing $P$ which has the same bonus as $P$ has in the original graph.  It then follows from a straightforward cost calculation that this gives a reducible subgraph.
Next consider the figure from the centre and the indicated partial ear decomposition.
In this case the gain of the decomposition is 40, and the total bonus of the ears involved is at most 40.
%If the length of $P$ is not 2 mod 3, then the indicated partial ear decomposition gives a reducible subgraph.  If $P$ has length 2 mod 3, then we have a smaller reducible subgraph given by the same partial ear decomposition without $Q$.

So, we may now assume that $G^{\Delta} \not\cong K_{3,3}$.  Since $G^\Delta$ is cyclically 4-edge-connected, $G^\Delta$ is not isomorphic to Prism. It follows that every vertex in $V(G^{\Delta}) \setminus V(C)$ has at most one neighbour in $V(C)$.  So, we may assume (without loss) that the configuration in the lemma extends to that appearing on the right in the above figure.  Let $e,e'$ be the edges of $G^{\Delta}$ associated with $R, R'$ and observe that $G^{\Delta} - \{e,e'\}$ is a subdivision of a 3-edge-connected graph.  It follows from this that the partial ear decomposition $R,R'$ extends to a reducible subgraph (the cost of the endpoint pattern at $v$ will be at most 10, and there is one newly formed ear with length 1 mod 3).  This completes the proof.
\end{proof}

\begin{lemma}
\label{tame2insq}
Let $C$ be a 4-cycle in $G^{\Delta}$ containing neither $w$ nor a vertex corresponding to a 112 triangle.  If the corresponding subgraph of $G^{\bullet}$
  contains an outer ear of length 2 mod 3, then this subgraph appears as follows:

\begin{myfigure}[h]
\centerline{\includegraphics[width=2.75cm]{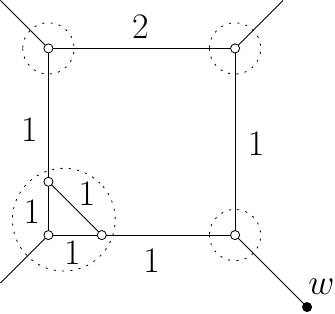}}
\end {myfigure}
\end{lemma}

\begin{proof} Let $H$ be the subgraph of $G^{\bullet}$ corresponding to $C$ and let $P_1, \ldots, P_4$ be the outer ears of $G^{\bullet}$ contained in $H$.  Note that by Lemma~\ref{outer2force} and our assumptions if $P_i$ has length 2 mod 3, then both ends of $P_i$ must be triads.  If all of $P_1, \ldots, P_4$ have length 2 mod 3, then $H$ is a cycle with length 2 mod 3 which is contractible (in particular it contradicts Lemma~\ref{easycont}).  For the other cases we will call upon the following figure.

\begin{myfigure}[htp]
\centerline{\includegraphics[height=2.75cm]{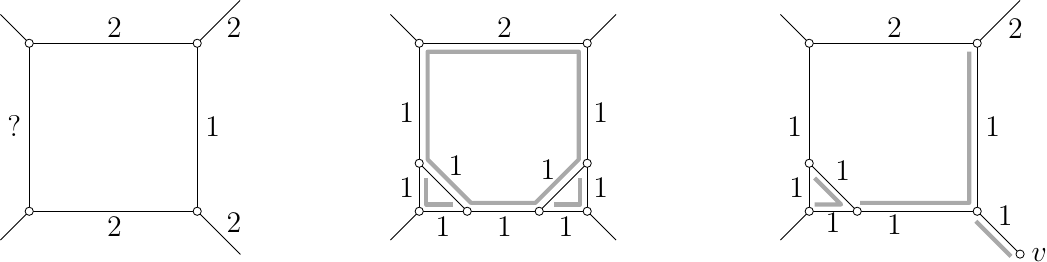}}
\end{myfigure}

If exactly three of $P_1, \ldots, P_4$ have length 2 mod 3 then we have the pattern on the left in the figure, and this gives a contradiction to Lemma~\ref{forbid12}.  If exactly two of $P_1, \ldots, P_4$ have length 2 mod 3  then either we have another contradiction to Lemma~\ref{forbid12} or a contradiction to the previous lemma.  So, we may now assume (without loss) that $P_1$ has length 2 mod 3 and $P_2, \ldots, P_4$ have length 1 mod 3.

If $H$ has no inner triangle, then $H$ is a cycle with length 2 mod 3 which is contractible.  If $H$ has two inner triangles, then $H$ is a contractible subgraph as shown in the middle.  Otherwise $H$ has exactly one inner triangle, and if we are not in the case
  indicated in the statement of the lemma, we must have the configuration on the right in the above figure (using Lemma~\ref{outer2force} and~\ref{even2}).
  We claim that the partial ear decomposition indicated by this figure extends to a reducible subgraph.  This follows from Lemma~\ref{clever_choice} (which allows us to have endpoint cost at $v$ at most 5), the fact that an ear of length 1 modulo 3 is created by the reduction  and a straightforward calculation.
\end{proof}

We are now ready to prove that $G^{\Delta}$ has at least 8 vertices.
Note that the structure of the proof of Lemma~\ref{notk33} is similar to
  the proof of Lemma~\ref{le:no222}, as $G^\Delta \cong K_{3,3}$ was the difficult case there;
  the details of the arguments are different, though.

\begin{lemma}
\label{notk33}
$G^{\Delta}$ is not isomorphic to $K_{3,3}$.
\end{lemma}

\begin{proof}
Suppose (for a contradiction) that $G^{\Delta} \cong K_{3,3}$ and
let $( \{v_1, v_2, w \}, \{u_1, u_2, u_3\} )$ be a bipartition of $G^{\Delta}$.
  First we consider the possibility that $G^{\bullet}$ contains a triangle of type 112.  In this case Lemma~\ref{tame112} and Lemma~\ref{even2} imply that there are exactly two usable outer ears with length 2 mod 3.  All of these possibilities are handled by the partial ear decompositions appearing in the following figure depicting $G^{\bullet}$ after removing the unusable ears.

\begin{myfigure}[htp]
\centerline{\includegraphics[height=3cm]{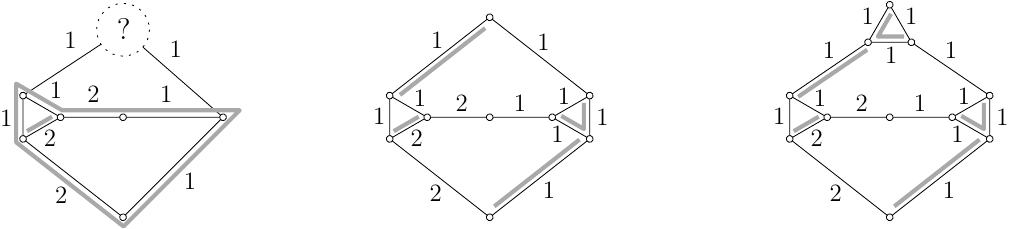}}
\end {myfigure}

In the first graph in the above figure, the indicated configuration is contractible. An easy calculation shows that the partial ear decomposition given
  in the third graph is reducible. Now consider the middle graph in the figure above. From top to bottom, let $v_1, v_2$ and $v_3$ denote the three
  neighbours of $w$, and let $P_i$ be the ear with ends $w$ and $v_i$. Finally, let $Q$~denote the marked ear that is incident with the type 111
  triangle and let $P$, $P'$~be the two marked ears contained in the 111 triangle.
The total gain of the given reduction is 48, while the bonus of the ears other than $P_1, P_2$ and $P_3$ is 45. By Lemma~\ref{triangle_types}, all the
  ears created by deleting an inner ear of length 2 mod 3 are equitable. Thus the reduction creates a new ear $P_3'$ containing $P_3$ and with the same
  bonus as $P_3$. Similarly, the bonus of the new ear containing $P_1$ is the same as the bonus of $P_1$, unless the ear created by removing $P$ and $P'$
  from $G^\bullet$ is inequitable. However, in this case $P, P', Q$ is a reducible configuration
  (in some variants we use Lemma~\ref{clever_choice} for removing~$Q$).

So we may assume that every inner triangle of $G^{\bullet}$ has type 111.  If there were an edge of $G^{\Delta}$ of residue 2 incident with $v_i$ for $i=1,2$, then Lemma~\ref{even2} would imply the existence of two such edges.  However, then $G^{\Delta}$ would have a 4-cycle containing these two edges which contradicts Lemma~\ref{tame2insq}.  It follows that every usable ear of $G^{\bullet}$ has length 1 mod 3.  Now by applying Lemma~\ref{all1insq} to the three 4-cycles of $G^{\Delta}$ not containing $w$ implies that we have one of the cases in the following figure.  Here we have depicted the graph obtained from $G^{\bullet}$ by removing the unusable ears, and all pictured ears have length 1 mod 3.

\begin{myfigure}[htp]
\centerline{\includegraphics[width=12cm]{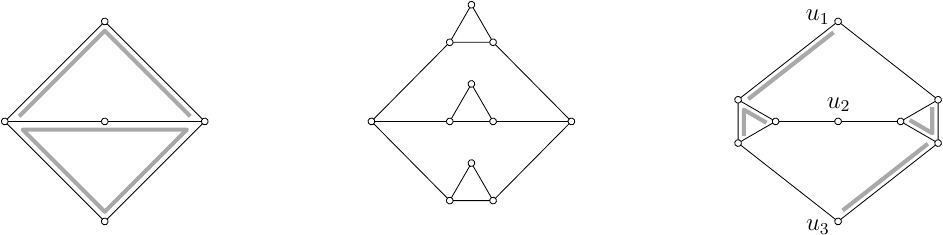}}
\end {myfigure}

In the first case we have a contractible subgraph as indicated by in the figure.  In the second one we have a contradiction to Lemma~\ref{forbid12}.  In the last case $u_1,u_2,u_3$ correspond to triads in $G^{\bullet}$ while $v_1,v_2$ correspond to triangles in~$G^{\bullet}$ (of type 111).  By the symmetry of $u_1,u_2,u_3$ we may assume (without loss) that either both edges $w u_1$ and $w u_3$ of $G^{\Delta}$ have nonzero residue, or both have zero residue.  In the former case, the indicated partial ear decomposition gives a reducible subgraph (in the resulting graph, the ear containing $u_1$ ($u_3$) will have the same (nonzero) length mod 3 as the ear of $G^{\bullet}$ with ends $w$ and $u_1$ ($u_3$)).   In the latter case, let the indicated partial ear decomposition be given by $Q',Q,P',P$ where $P$ ($Q$) is the ear with end $u_1$ ($u_3$) and $P$ and $P'$ intersect.  Choose an ear labelling of these ears with gain $\ge 48$ and consider the weighted graph $G'$ obtained from $G$ by performing the associated removals.
If, for both $i=1$ and~$i=3$, either the ear of $G'$ containing $u_i$ is inequitable or the ear of $G^\bullet$ between $w$ and $u_i$ is equitable, then we have a reducible configuration.
Otherwise, we may assume that the ear of $G'$ containing $u_3$ is equitable and the ear of $G^\bullet$ between $w$ and $u_3$ is inequitable.  However, in this case the partial ear decomposition $P', P$ is a reducible subgraph (removing these two ears as before will yield a weighted graph with an inequitable ear of length 0 mod 3 with ends $u_2$ and $u_3$), which is a contradiction.
\end{proof}

\subsection[Removing Adjacent Vertices of $G^{\Delta}$]{Removing Adjacent Vertices of $\mathbf{G^{\Delta}}$}

\begin{lemma}
\label{2nonequitable}
Assume that $v_1, \ldots, v_4$ is a cyclic ordering of a 4-cycle $C$ of $G^{\bullet}$ and assume that every vertex adjacent to one of $v_1, \ldots, v_4$ has degree 3.  Then there exist ear labellings $\psi : \{ v_1 v_2 \} \rightarrow \mathbb{Z}_3$ and $\psi' : \{ v_3 v_4 \} \rightarrow \mathbb{Z}_3$ each with gain $\geq$ 8, so that in the weighted graph obtained from $G^{\bullet}$ by the $\psi$-removal of the ear $v_1 v_2$ and then the $\psi'$-removal of the ear $v_3 v_4$, the ears given by the vertex sequences $u_1, v_1, v_4, u_4$ and $u_2, v_2, v_3, u_3$ are both inequitable (where $u_i$ is the neighbour of $v_i$ outside $C$).
\end{lemma}

\begin{proof}
First suppose that the function $\mu_G$ satisfies $\mu_G(v_1) \neq 0$.  In this case we may choose $\psi : \{ v_1 v_2 \} \rightarrow \mathbb{Z}_3 \setminus \{0\}$ so that $\partial \psi (v_1) = \mu_G(v_1)$.  Let $G'$ be the weighted graph obtained from $G^{\bullet}$ by the $\psi$-removal of the path given by $v_1, v_2$.  Note that $u_1, v_1, v_4$ is the vertex sequence of an ear in $G'$ and every ear labelling of this ear has support of size 0 or 2.  Now we may apply Lemma~\ref{clever_choice} to $v_3$ to choose an ear labelling $\psi' : \{ v_3 v_4 \} \rightarrow \mathbb{Z}_3 \setminus \{0\}$ so that in the weighted graph obtained from $G'$ by the $\psi'$-removal of the path given by $v_3, v_4$, the ear given by $u_2, v_2, v_3, u_3$ is inequitable.  The ear given by $u_1, v_1, v_4, u_4$ will also be inequitable, thus finishing our proof in this case.

By the above argument we may now assume that $\mu_G(v_i) = 0$ for $1 \le i \le 4$.  Assume (without loss) that the edges of $G^{\bullet}$ are oriented so that $v_1 v_2$ ($v_3 v_4$)  is directed from $v_1$ to $v_2$ (from $v_3$ to $v_4$) and so that $u_1, v_1, v_4, u_4$ and $u_2, v_2, v_3, u_3$ are vertex sequences of directed paths.  Now define $\psi : \{ v_1 v_2 \} = 1$ and $\psi' : \{ v_3 v_4 \} = 1$ and consider the weighted graph obtained from the $\psi$-removal of the ear $v_1 v_2$ and then the $\psi'$-removal of $v_3 v_4$.  In this weighted graph  the paths $u_1, v_1, v_4, u_4$ and $u_2, v_2, v_3, u_3$ form ears with the property that every ear-labelling assigns the first and last edge the same value.  It follows that these ears are inequitable, as desired.
\end{proof}

A 3-dimensional cube and Wagner's graph will play a special role in our paper. The former one will be denoted by $\mbox{Cube}$, the latter one by
  $\mbox{$V_8$}$. These graphs will be dealt with in the next section. Now we will find several reducible configurations
  that involve removing two adjacent vertices in certain (typical) cases. The remaining cases will be dealt with at the end of the proof,
  after we get to understand structure of 4-edge-cuts in~$G^{\Delta}$.

\begin{lemma}
\label{deleteadj}
Let $u,v$ be adjacent vertices in $G^{\Delta}$ and assume the following:
\begin{itemize}
\item both of $u,v$ correspond to triangles, or both correspond to triads in $G^{\bullet}$,
%\item $w$ is not a neighbour of $u$ or $v$ in $G^{\Delta}$,
\item neither~$u$ nor~$v$ is adjacent to~$w$ in~$G^{\Delta}$,
\item all edges of $G^{\Delta}$ incident with $u,v$ have residue 1,
\item $G^{\Delta} - \{u,v\}$ is cyclically 3-edge-connected.
\end{itemize}
Then $G^{\Delta}$ is not isomorphic to Cube or $V_8$.

Furthermore, if no 4-cycle of $G^{\Delta}$ contains $w$ and either $u$ or $v$, then there is a unique 4-cycle containing the edge $uv$ in $G^{\Delta}$ and $u,v$ appear in a configuration of $G^{\bullet}$ as follows:

\begin{myfigure}[htp]
\centerline{\includegraphics[width=5cm]{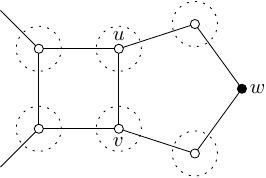}}
\end{myfigure}

\end{lemma}

\begin{proof}
Let $H$ be the subgraph of $G^{\bullet}$ corresponding to the subgraph in $G^{\Delta}$ induced by all edges adjacent to $u$ or to $v$.
By Lemmas~\ref{triangle_types} and~\ref{le:no222}, if $u$ and~$v$ correspond to triangles in $G^\bullet$, then they are of type 111. As a first attempt toward finding a reducible subgraph, consider the partial ear decomposition $P_1, \ldots, P_k$ of $G^{\bullet}$ with $\cup_{i=1}^k P_i = H$ indicated by the following figure.

\begin{myfigure}[h]
\centerline{\includegraphics[width=10cm]{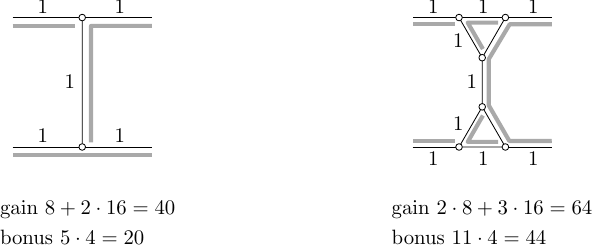}}
\end{myfigure}

First suppose that the edge $uv \in E( G^{\Delta} )$ is not contained in a 4-cycle of~$G^{\Delta}$. In particular, this means that $G^{\Delta}$ is not isomorphic to Cube or $V_8$, so we may additionally assume that no 4-cycle of $G^{\Delta}$ contains $w$ and either $u$ or $v$. In this case we will show that $H$ may  be extended to a reducible subgraph using the endpoint patterns and their associated costs.  This is immediate if every endpoint pattern has cost at most 5.  In the other case there must be one or more endpoint patterns of type $Y_{12}$.  If there are at most two such endpoint patterns, then by Lemma~\ref{clever_choice} we may arrange for a cost of 3 at one of them, and again we are done.  Otherwise, we may assume (without loss) that in the graph $G^{\Delta}$ the vertex $u$ has neighbours $\{v,u',u''\}$ where both $u'$ and $u''$ are incident with two edges of residue 1 and one of residue 2.  However
in this case Lemma~\ref{forbiddentriads} implies that both $u'$ and~$u''$ will be adjacent to $w$ in $G^{\Delta}$, contradicting that no 4-cycle of $G^{\Delta}$ contains $w$ and either $u$ or $v$.

We may now assume that $uv$ is contained in at least one 4-cycle of $G^{\Delta}$. We cannot compute our endpoint costs independently as the newly formed graph will have ears which ``utilize'' more
  than one of these endpoints.   To get a handle on these possibilities we will consider the possible ears of the graph $G^{\Delta} - \{u,v\}$.  It
  follows from Lemma~\ref{notk33} and cyclic 4-edge-connectivity that  $G^{\Delta} - \{u,v\}$ is not a cycle.  If $G^{\Delta} - \{u,v\}$ has an ear of length 4 whose interior
  vertices are $\{x_1, x_2, x_3\}$, then $d( \{u,v, x_1, x_2, x_3 \} ) = 3$, but then a combination of the cyclic 4-edge-connectivity of $G^{\Delta}$
  and Lemma~\ref{notk33} give us a contradiction.  So, every ear of $G^{\Delta} - \{u,v\}$ has length at most $3$.  Next suppose that $P$ is an ear of
  $G^{\Delta}- \{u,v\}$ with length 3 and note that $P$ must have one interior vertex adjacent to $u$ in $G^{\Delta}$ and the other adjacent to $v$
  in $G^{\Delta}$ (thus giving us a 4-cycle using the edge $uv$).  The following figure assigns costs to the subgraph of $G^{\bullet}$ associated
  with~$P$; we will call these costs \emph{side pattern costs}.  (In the figure, the ears entering each configuration on the right are part of $H$ and are therefore getting removed.)

\begin{myfigure}[htp]
\centerline{\includegraphics[width=13cm]{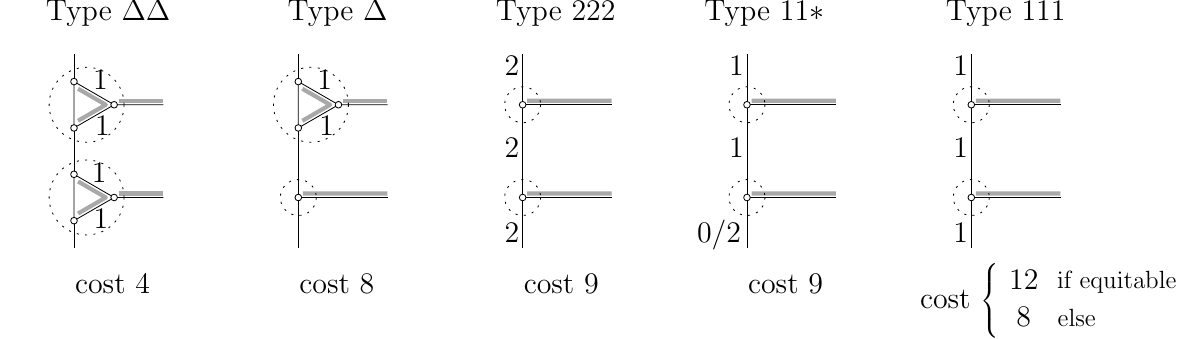}}
\end{myfigure}

For the three rightmost configurations of the figure, the cost is an upper bound on the drop in bonus of the ears associated with $P$ in $G^{\bullet}$ to the bonus of the single ear $P$ in the resulting graph.  In the two leftmost configurations, we need to extend our partial ear decomposition by removing some additional ear(s) from the subgraph of $G^{\bullet}$ associated with $P$ in order to maintain our connectivity.  Here the cost indicates an upper bound on the drop in bonus minus the gain of the indicated partial ear decomposition. 

First suppose that $uv$ is contained in two squares of $G^{\Delta}$.  If neither associated side pattern is type 111, then each has cost at most 10 giving us a
  reducible configuration.  So, we may assume (without loss) that at least one side pattern~$P$ has type 111.  If $u,v$ both correspond to inner triangles in
  $G^{\bullet}$, then by applying Lemma~\ref{clever_choice} (and the ear decomposition from the first figure in the proof) we may choose a function $\psi : E(H)
  \rightarrow \mathbb{Z}_3$ with gain at least 64 so that the ear of the resulting graph corresponding to $P$ is inequitable (i.e., we can arrange to have the side pattern associated with $P$ cost 8 instead of 12).  It then follows that we have found a removable subgraph.  Next suppose that $u,v$ correspond to triads in $G^{\bullet}$ where $u$ has neighbour set $\{v,u',u''\}$ and $v$ has neighbour set $\{u, v', v''\}$ and assume further that $u', v'$ are the interior vertices of $P$.
By Lemma~\ref{le:length1}, all the neighbours of $u,v,u',v'$ have degree 3.
It now follows from Lemma~\ref{2nonequitable} applied to the 4-cycle $u,v,u',v'$ that we may choose ear labellings $\psi_1: \{ uu' \} \rightarrow \mathbb{Z}_3$ and $\psi_2 : \{ vv' \} \rightarrow \mathbb{Z}_3$ of gain 8 so that in the weighted graph $G'$ obtained by the $\psi_1$- and $\psi_2$-removal both the ear containing $u'$ and $v'$ and the ear containing $u''$ and $v''$ are inequitable (and length 0 mod 3).  Let $\psi_3 $ be an ear labelling of the ear of $G'$ containing $u''$ and $v''$ with gain $\ge 24$.  Now $\psi_1 \cup \psi_2 \cup \psi_3$ is a function on $H$ with gain 40 and we have arranged that the cost of the side pattern $P$ is at most 8, thus giving us a reducible configuration.

So, we may assume that $uv$ is contained in exactly one square of $G^{\Delta}$ (in particular, $G^{\Delta}$ is not isomorphic to Cube) with vertices cyclically ordered as $u, u', v', v$, and we let the neighbour sets of $u$ and $v$ be $\{v, u', u'' \}$ and $\{u, v', v''\}$.  If the costs associated with the endpoint patterns at $u''$ and $v''$ are at most 5, then we will again have a reducible subgraph (by using the above procedure we may arrange that the cost of the side pattern will be at most 10).  So, we may assume without loss that $u''$ has endpoint pattern of type $Y_{12}$.  However, in this case we may apply Lemma~\ref{clever_choice} to arrange for a cost of 3 at this endpoint.  It follows that we will have a reducible configuration unless both $u''$ and $v''$ have endpoint pattern of type $Y_{12}$ and we have a side pattern of cost 12.  In this case Lemma~\ref{forbiddentriads} implies that both $u''$ and $v''$ must be adjacent to $w$ which gives us the following configuration in $G^{\Delta}$.

\begin{myfigure}[htp]
\centerline{\includegraphics[width=5cm]{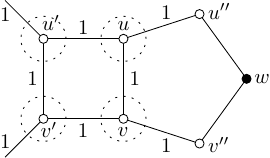}}
\end{myfigure}

We may now assume $G^{\Delta} \cong V_8$ as otherwise we have nothing left to prove.
Up to symmetry, there are only two ways how the configuration in the picture above
can be included in~$V_8$: $w$ is adjacent either to~$u'$ or to~$v'$. We assume the latter.
However, then $u, u'$ are adjacent vertices of $G^{\Delta}$ incident only with edges of residue 1 so that $uu'$ is contained in two 4-cycles.  It now follows from
  applying the above argument to $u,u'$ instead of $u,v$ that we have a removable subgraph (the connectivity condition follows from the assumption
  $G^{\Delta} \cong V_8$).  This final contradiction completes the proof.
\end{proof}

\subsection[Cube and $V_8$]{Cube and $\mathbf{V_8}$}

Next we take care of a couple more particular small graphs.

\begin{lemma}
\label{notv8cube}
$G^{\Delta}$ is not isomorphic to Cube or $V_8$.
\end{lemma}

\begin{proof}
  Assume the contrary.
  First suppose that $G^{\bullet}$ contains a triangle of type~112.  In this case Lemma~\ref{tame112} implies that there is a 4-cycle $C$ in $G^{\Delta}$ which contains the vertex $w$ and also contains a vertex corresponding to the unique type 112 triangle.  Let $e_0, \ldots, e_3$ be the edges of $G^{\Delta}$ containing exactly one vertex in $V(C)$ and assume that $e_0$ is incident with $w$.  It now follows from Lemma~\ref{tame112} that $e_1, e_2, e_3$ all have residue 1.  Let $C'$ be the 4-cycle given by $G^{\Delta} - V(C)$ and assume (without loss) that $C'$ has edges $f_0, \ldots, f_3$ where $f_i$ is adjacent to the edges $e_i$ and $e_{i+1}$ treating the indices mod 4.
  (Note that we do not assume that the order of $e_0$, \dots, $e_3$ corresponds to their order along~$C$.)
  It now follows from Lemma~\ref{tame2insq} and Lemma~\ref{outer2force} that $f_i$ has residue 1 for $0 \le i \le 3$.  By applying Lemma~\ref{deleteadj} to the ends of $f_1$ and $f_2$ we deduce that each of these edges has one end corresponding to a triad in $G^{\bullet}$ and the other corresponding to a 111 triangle in $G^{\bullet}$.  It now follows from Lemma~\ref{all1insq} applied to $C'$ that we may choose a vertex $x \in V(C')$ so that the following holds:
\begin{itemize}
\item $x$ is not adjacent to $w$,
\item $x$ corresponds to a triad in $G^{\bullet}$,
\item Both neighbours of $x$ in $V(C')$ correspond to triangles of type 111.
\end{itemize}

Let $H$ be the subgraph of $G^{\bullet}$ consisting of the three ears with endpoint $x$ and note that $H$ has an ear decomposition of gain 24.  The bonus of $H$ is 12, and the three endpoint patterns we get by removing it will have two of type $\Delta_{11}$ and one of type $\Delta_{11}'$, $Y_{0*}$, $Y_{12}$, or $Y_{22}$. (The endpoint pattern $Y_{11}$ does not appear, because we are assuming that $C$ contains a type 112 triangle, and so has the structure given by Lemma~\ref{tame112}.)  In any of these cases (possibly by employing Lemma~\ref{clever_choice}) we may arrange that the total of these endpoint costs is at most 12, thus giving us a removable subgraph.

So we may now assume that $G^{\bullet}$ has no triangle of type 112.  We claim that every edge of $G^{\Delta}$ not incident with $w$ has residue 1.
  To see this, first suppose $G^{\Delta}$ is isomorphic to Cube.  If $\overline{w}$ is the vertex of $G^{\Delta}$ which is the antipode of $w$, then
  all three edges of $G^{\Delta}$ incident with $\overline{w}$ must have residue 1 (otherwise by Lemma~\ref{even2} there would be two of residue 2 and
  this would violate Lemma~\ref{tame2insq} applied to the 4-cycle containing those two edges).  Now if there is an edge of $G^{\Delta}$ not incident
  with either $w$ or $\overline{w}$ of residue 2, this would give us a violation of Lemma~\ref{tame2insq}.  Next suppose that $G^{\Delta} \cong V_8$.
  Let $C_1, C_2$ be the two 4-cycles of $G^{\Delta}$ which do not contain $w$, and let $\{ e \} = E( C_1 ) \cap E(C_2)$.  If $e$ has residue 2, then
  Lemma~\ref{even2} implies that this must also be true for another edge adjacent to $e$, which gives us a contradiction with Lemma~\ref{tame2insq}.
  So, $e$ must have residue 1.  If $e$ has an endpoint $v$ that is incident with another edge $f$ of residue 2, then $v$ must also be incident with
  another edge $f' \neq f$ of residue 2 (by Lemma~\ref{even2}).  However, now applying Lemma~\ref{tame2insq} to one of $C_1$ or $C_2$ gives us a contradiction.  So, every edge of $G^{\Delta}$ adjacent to $e$ must also have residue 1.  Another application of Lemma~\ref{tame2insq} implies that all edges in $C_1 \cup C_2$ must have residue 1.  Now Lemma~\ref{even2} implies that all edges of $G^{\Delta}$ not incident with $w$ have residue 1, as desired.

Now we will return to considering the cases of Cube and $V_8$ simultaneously.  As we did above, choose 4-cycles $C_1, C_2$ of $G^{\Delta}$ which do not contain $w$.  In either case there is a unique edge $e$ which is contained in both $E(C_1)$ and $E(C_2)$.  For $i=1,2$ let $v_i$ be the unique vertex in $C_i$ for which there exists an edge with ends $w$ and $v_i$.  It now follows from Lemma~\ref{deleteadj} that for $i=1,2$  both edges of $C_i$ not containing $v_i$ have the property that one end corresponds to a triad in $G^{\bullet}$ while the other end corresponds to a triangle.  Thus, by Lemma~\ref{all1insq}, each of $C_1$ and $C_2$ has two nonadjacent vertices corresponding to triads and two nonadjacent vertices corresponding to 111 triangles.  It now follows that one end of $e$ is a triad which is adjacent (in $G^{\Delta}$) to three vertices each of which corresponds to a 111 triangle.  However this contradicts Lemma~\ref{forbid12},  thus completing the proof.
\end{proof}

\subsection{Pushing a 4-Edge-Cut}

For the purposes of maintaining connectivity, earlier in our argument we pushed a 3-edge-cut in the original graph $G$ (i.e., we chose a 3-edge-cut $\delta(X)$ with $X$ minimal subject to some conditions on $X$).  This gave us the graph $G^{\bullet}$ which we have investigated at length. However, in order to complete the proof, we will need some slightly stronger connectivity properties.  For this purpose we next consider 4-edge-cuts in the graph $G^{\Delta}$ and look to find an extreme one with some useful properties.

%\begin{observation}
%\label{2vtx2edge}
%Let $H$ be a cyclically 4-edge-connected cubic graph, let $u,v$ be adjacent vertices and suppose that $H - \{u, v\}$ is not cyclically 3-edge-connected.  Then there exist distinct edges $e_1,e_2 \neq uv$ incident with $u$ and distinct edges $f_1,f_2 \neq uv$ incident with $v$ so that $H - \{e_1, f_1\}$ and $H - \{ e_2, f_2 \}$ are not cyclically 3-edge-connected.  Furthermore, every 4-cycle of $H$ containing $uv$ uses both $e_i$ and $f_i$ for some $i = 1,2$.
%\end{observation}

%\begin{proof}
%Let $\{Y_1,Y_2\}$ be a partition of $V(H - \{u,v\} )$ corresponding to a cyclic edge cut of size at most 2.  Define $N_u = N_H(u) \setminus \{v\}$ and $N_v = N_H(v) \setminus \{u\}$.  If some $Y_i$ contains at least three vertices from $N_u \cup N_v$ then we may extend $\{Y_1,Y_2\}$ to a partition of $V(H)$ which gives a cyclic edge cut of size $\le 3$,  contradicting our assumptions.  Similarly, if some $Y_i$ contains $N_u$ and the other contains $N_v$, then $H$ has an edge cut of size $\le 3$ separating cycles, and we have a contradiction.  For $i=1,2$ let $e_i$ ($f_i$)  be the unique edge of $H$ with one end $u$ ($v$) and the other in $Y_i$.  Then $H - \{e_1, f_1\}$ and $H - \{ e_2, f_2 \}$ are not cyclically 3-edge-connected.  If $H$ has a 4-cycle containing $uv$ containing $e_1$ and $f_2$ then $\{ Y_1 \setminus V(C), Y_2 \cup V(C) \}$ gives an edge cut in $H$ of size $\le 3$ separating cycles and this is contradictory.
%\end{proof}

We say that a 4-cycle $C \subseteq G^{\Delta}$ has property $\alpha$ or $\beta$ if $C$ does not contain any vertex which is in a 4-cycle with $w$ (so in particular, $w \not\in V(C)$ and $C$ does not contain any vertex corresponding to a type 112 triangle), and the edges of~$C$ may be put in cyclic order $e_1, e_2, e_3, e_4$ so that the corresponding property indicated in the following table is satisfied:

\begin{tabular}{|c|p{4.25in}|}
\hline
Property	&	Required connectivity\\
\hline
$\alpha$	&	$N(w) \cap V(C) = \emptyset$ and $G^\Delta - \{ e_1, e_3 \}$ is cyclically 3-edge-connected.	\\
\hline
$\beta$	&	$|N(w) \cap V(C)| = 1$ and both $G^\Delta - \{ e_1, e_3 \}$ and $G^\Delta - \{ e_2, e_4 \}$ are cyclically 3-edge-connected. \\
\hline
\end{tabular}

\begin{lemma}
\label{alphabetagood}
Either $G^{\Delta}$ contains a 4-cycle of type $\alpha$ or $\beta$, or there exists a set $Z \subseteq V(G^{\Delta}) \setminus \{w\}$ satisfying all of the following:
\begin{enumerate}
\item if $C$ is a 4-cycle containing $w$, then $V(C) \cap Z = \emptyset$,
\item $Z$ induces a graph containing a cycle,
\item $d(Z) = 4$,
\item $Z$ induces a subgraph of minimum degree $\ge 2$ and girth $\ge 5$ in which no two degree 2 vertices are adjacent, and
\item no edge with both ends in $Z$ is in a cyclic 4-edge-cut of $G^{\Delta}$.
%\item If $y,z \in Z$ are adjacent, then $G^{\Delta} - \{y,z\}$ is cyclically 3-edge-connected.
\end{enumerate}
\end{lemma}

\begin{proof}
  The graph $G^{\Delta}$ is a cyclically 4-edge-connected cubic graph (see Lemma~\ref{cyc4ec} and Corollary~\ref{cyc4eccor}),
  and by Lemmas~\ref{notk4}, \ref{notk33} and~\ref{notv8cube} it follows that $|V(G^{\Delta})| \ge 10$.  We claim that there exists $Z_0 \subseteq V(G^{\Delta}) \setminus \{w\}$ satisfying the following:
\begin{itemize}
\item if $C$ is a 4-cycle containing $w$, then $V(C) \cap Z_0 = \emptyset$,
\item $Z_0$ induces a graph containing a cycle,
and
\item $d(Z_0) \le 4$.
\end{itemize}
If there is no 4-cycle containing $w$, then the complement of $w$ is a set with the requisite properties.  If there is a unique 4-cycle $C$ containing $w$, then the set $Z_0 = V(G^{\Delta}) \setminus V(C)$ satisfies these conditions.  If there are exactly two 4-cycles $C_1, C_2$ containing $w$, then $C_1$ and $C_2$ must share exactly  one edge (by Corollary~\ref{cyc4eccor} and Lemma~\ref{notk33}) and the set $Z_0 = V(G^{\Delta}) \setminus (V(C_1) \cup V(C_2))$ has the required properties.  Finally, if there are three distinct 4-cycles $C_1, C_2, C_3$ containing $w$, then each pair of these must share at least one edge, so $d( V(C_1) \cup V(C_2) \cup V(C_3) ) \le 3$.  In this case our connectivity assumptions imply that there is at most one vertex not in $V(C_1) \cup V(C_2) \cup V(C_3)$ and this violates $|V(G^{\Delta})| \ge 10$.  Therefore, there exists a set $Z_0$ satisfying the desired properties.

Now we choose a set $Z$ satisfying the above properties so that $Z$ contains a minimum number of neighbours of $w$, and subject to this, $Z$ has minimum size.  If $d(Z)  < 4$, then our connectivity implies that $V(G^{\Delta}) \setminus Z = \{w\}$, but then we may decrease $Z$ by removing a neighbour of $w$ to get a set which contradicts our choice.  It follows that we must have $d(Z) = 4$.  So, our set $Z$ satisfies the first three properties in the statement of the lemma.

\paragraph{Suppose that $Z$ does not induce a 4-cycle.}
In this case the minimality of $Z$ immediately implies that the fourth property holds.  Next suppose (for a contradiction) that there is an edge $yz$
  with $y,z \in Z$ so that $yz$ lies in a cyclic 4-edge-cut $\delta(Y)$ of $G^{\Delta}$. Note that since $Y$ separates cycles, $|Y|>2$ and $|V(G^{\Delta})\setminus Y|> 2$.   We may assume (by possibly replacing $Y$ with its
  complement) that $w \not\in Y$.  If $Y \setminus Z$ is empty, then $Y$ contradicts the choice of $Z$.  It follows from $d(Y) = 4 = d(Z)$ and basic
  considerations that all of $d(Z \cap Y)$, $d(Z \setminus Y)$, $d(Y \setminus Z)$, and $d(Z \cup Y)$ have the same parity. Note that
  all of the sets $Y \cap Z$, $Y \setminus Z$, $Z \setminus Y$, $Y \cup Z$ are non-empty.
  By uncrossing we also have
\[
  \max \bigl( d(Z \cap Y) + d(Z \cup Y) \, , \, d(Z \setminus Y) + d(Y \setminus Z)\bigr) \le d(Z) + d(Y) = 8
\]
If the parity of our four parameters is odd, then one of $d(Z \cap Y)$ and $d(Z \cup Y)$ is equal to 3 and one of $d(Z \setminus Y)$ and $d(Y \setminus
  Z)$ is equal to 3.  Since the only 3-edge-cuts of $G^{\Delta}$ separate one vertex from the rest, the only way for this to happen is if
  $V(G^{\Delta}) \setminus Z$ has size two.  In this case there is no 4-cycle containing $w$, and the set $Y$ contradicts the choice of $Z$ (since it
  has at most as many neighbours of $w$ as $Z$ and fewer vertices than $Z$).   So, we may assume that the parity of our four edge-cut sizes is even.
  Now our uncrossing equations imply that $d(Z \cap Y) = d(Z \setminus Y) = 4$.  If one of these sets induces a graph with a cycle, then it contradicts
  the choice of $Z$.  Otherwise, both have size 2, but then $Z$~must induce a 4-cycle, thus contradicting our assumption.
  It follows that no edge $yz$ with $y,z\in Z$ is contained in a cyclic 4-edge-cut, thus verifying the final property.

\paragraph{Now let us assume that $Z$ induces a 4-cycle $C$.} Let edges of~$C$ be, in cyclic order, $e_1, e_2, e_3, e_4$.  As a first case, we shall assume that $Z$
  does not contain a neighbour of $w$.  Suppose (for a contradiction) that there exists a cyclic 4-edge-cut $\delta(Y)$ containing $\{e_1, e_3\}$ and
  another cyclic 4-edge-cut $\delta(Y')$ containing $\{e_2,e_4\}$. If one of the sets
  $Y \cap Y'$, $Y \setminus Y'$, $Y' \setminus Y$, $Y \cup Y'$ is empty, then since $(e_1, e_2, e_3, e_4)$ is a cycle and $d(Y)=d(Y')=4$, then in fact two of these sets are empty and $Y'=Y$ or $Y' = V(G^{\Delta}) \setminus Y$, which leads to a contradiction. So, we may assume that all four of these sets are nonempty. By basic considerations, all of $d(Y \cap Y')$, $d(Y \setminus Y')$, $d(Y'
  \setminus Y)$, and $d(Y \cup Y')$ have the same parity.  If they are all odd, then one of $d(Y \cap Y')$ and $d(Y \cup Y')$ must equal three and one of
  $d(Y \setminus Y')$ and $d(Y' \setminus Y)$ must equal three.  Since these size three edge-cuts must separate one vertex from the rest, this gives a
  contradiction to the assumption that both $\delta(Y)$ and $\delta(Y')$ are cyclic edge-cuts. It now follows from the uncrossing inequalities that
  $d(Y \cap Y') = d(Y \setminus Y') = d(Y' \setminus Y) = d(Y \cup Y') = 4$.

  By the existence of our 4-cycle $C$, each of the sets $Y \cap Y'$, $Y \setminus Y'$, $Y' \setminus Y$, $V(G^{\Delta}) \setminus (Y \cup Y')$ must induce a graph with a vertex of degree $\le 1$.  Now removing this vertex from the
  corresponding set decreases the size of the associated edge-cut to 3.  It follows from this that $|Y \cap Y'| = |Y \setminus Y'| = |Y' \setminus Y| =
  |V(G^{\Delta}) \setminus (Y \cup Y')| = 2$, giving us a contradiction (since $|V(G^\Delta)|\geq10$).  So, we may assume, without loss of generality,
  that there is no 4-edge-cut containing $\{e_1, e_3\}$ and thus $C$ has property $\alpha$.

The remaining case is when $Z$ induces a 4-cycle $C$ and $C$ contains a vertex which is a neighbour of $w$.  Assume that the vertices of $C$ have cyclic
  order $v_1, \ldots, v_4$ and that $v_1$ is a neighbour of $w$.  Suppose (for a contradiction) that there is a cyclic 4-edge-cut $\delta(Y)$ with
  $v_2,v_3 \in Y$ and $v_1,v_4 \not\in Y$.  Note that in this case $w \not\in Y$, otherwise $d(Y\cup \{v_1\})=3$, which implies that
  $|V(G^\Delta)\setminus Y|=2$, a contradiction.  If $Y$ does not contain a neighbour of $w$, then $Y$ contradicts our choice of $Z$.
  (Note that if some 4-cycle containing~$w$ intersects~$Y$, then we get a contradiction with cyclic 4-edge-connectivity.)
  So, we may assume that $Y$ contains at least one neighbour of $w$.  However, in this case $d(Y \cup \{ w, v_1, v_4 \}) \le 3$ so the graph induced by
  $V(G^{\Delta}) \setminus Y$ is a 4-cycle containing $w$ and $v_1 \in Z$.  This gives us a contradiction to the choice of $Z$.  It follows that there is no
  4-edge-cut separating cycles containing $v_1 v_2$ and $v_3 v_4$.
  A symmetric argument shows that there is no cyclic 4-edge-cut $\delta(Y')$ with
  $v_3,v_4 \in Y'$ and $v_1,v_2 \not\in Y'$, and thus $C$ has property $\beta$.
\end{proof}

\subsection{Completing the Proof}

We will now complete the proof of our workhorse lemma.  The only additional thing we require is the following simple observation concerning cyclic 4-edge-cuts.  Here the first property is a direct consequence of the assumption, and the other two properties follow from the first.

\begin{observation}
  \label{obs4cuts}
Let $H$ be a cyclically 4-edge-connected cubic graph, let $v_1 v_2 \in E(H)$ and assume that $H - \{ v_1, v_2 \}$ is not cyclically 3-edge-connected.
Then the following is true: 
\begin{itemize}
\item There is a partition of $V(H)$ into $\{ Y_1, \{v_1,v_2\}, Y_2 \}$ so that $\delta(Y_i)$ is a cyclic 4-edge cut for $i=1,2$ and so that
there is exactly one edge between $v_i$ and $Y_j$ for every $1 \le i, j \le 2$.
\item Every edge adjacent to the edge $v_1 v_2$ is contained in a cyclic 4-edge-cut.
\item If there is a 4-cycle with vertex sequence $v_1, v_2, v_3, v_4$,
	then both $v_3$ and~$v_4$ are in~$Y_1$ or both are in~$Y_2$.
	Consequently, $H$ has a cyclic 4-edge cut containing the edges $v_1 v_4$ and $v_2 v_3$.
\end{itemize}
\end{observation}

\begin{proof}[Proof of Lemma \ref{workhorse}]
We will split the proof into cases based on the possible outputs of Lemma~\ref{alphabetagood}.  In each of these cases we will reach a contradiction, and this will complete the proof of our theorem.

%\bigskip

\paragraph{Case 1: $G^{\Delta}$ contains a 4-cycle $C$ of type $\alpha$.}

%\smallskip

\begin{figure}[htp]
\centerline{\includegraphics[height=3cm]{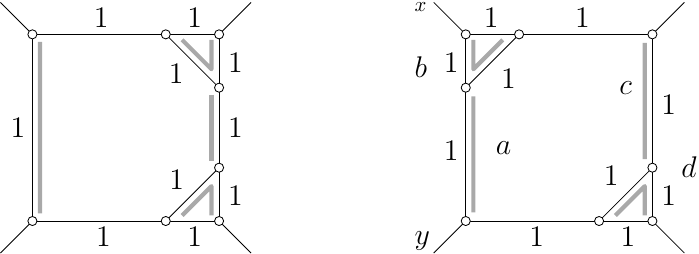}}
\caption{Reducing Squares}
\label{reducesqares}
\end {figure}

By assumption, $C$ cannot contain a vertex corresponding to a 112 triangle.  If there is an edge of $C$ of residue 2, then we get an immediate
  contradiction to Lemma~\ref{tame2insq}.  So, every edge of $C$ has residue 1 and by Lemma~\ref{even2} all edges incident with the vertices of $C$ have
  residue 1. Assume that the vertices of $C$ are cyclically ordered $v_1, \ldots, v_4$ so that $G^{\Delta} - \{v_1 v_2, v_3 v_4 \}$ is cyclically
  3-edge-connected.  It now follows from the previous observation that both $G^{\Delta} - \{ v_1, v_4 \}$ and $G^{\Delta} - \{ v_2, v_3 \}$ are
  cyclically 3-edge-connected.
  If $v_1$ and~$v_4$ both correspond to either triads or triangles in $G^{\Delta}$ and also $v_2$ and~$v_3$ both correspond
  to either triads or triangles in $G^{\Delta}$, then we can apply Lemma~\ref{deleteadj} to $v_1, v_4$ and also to $v_2, v_3$.
  This means, that the exceptional configuration in the statement of Lemma~\ref{deleteadj} appears in both cases.
  Consequently, all of $v_1,\ldots,v_4$ are in distance at most~2 from $w$, thus two of them are adjacent to the same neighbour
  of~$w$. Hence, there is either a triangle in~$G^{\Delta}$ (contradiction with the cyclic edge-connectivity), or
  a 4-cycle through $v_1 v_4$, or through $v_2 v_3$, different from~$C$ (contradiction with Lemma~\ref{deleteadj}).

  It follows that we may assume (without loss) that $v_1$ corresponds to a 111 triangle and $v_4$ corresponds to a triad.  It now follows from
  Lemma~\ref{all1insq}
  that the subgraph of $G^{\bullet}$ corresponding to $C$ is given by one of the configurations Figure~\ref{reducesqares} above.
  Here all of the ears in the figure will have length 1 mod 3 by our assumptions on $w$.  In both cases we have indicated a partial ear decomposition
  (on the right, we may remove them in order $a$, $b$, $c$, $d$). 
  These decompositions yield a reducible subgraph (since two of the new ears will have length 1 mod~3), thus completing the proof of this case.

%\bigskip

\paragraph{Case 2: $G^{\Delta}$ contains a 4-cycle of type $\beta$.}

%\smallskip

Let $C$ be a 4-cycle of type $\beta$ with vertices $v_1, \ldots, v_4$ as in the definition of this type, where the neighbour of $w$ in $C$ is $v_1$.  If
  $C$ has an edge of residue 2, then the subgraph corresponding to~$C$ must be as given by Lemma~\ref{tame2insq}.  However, in this case we have a reducible subgraph indicated by the following figure (in the weighted graph obtained by removing these ears to achieve the appropriate bonus, the ear containing $Q$ will have the same bonus as $Q$ had in $G^{\bullet}$ by Lemma~\ref{outer2force}).

\begin{myfigure}[htp]
\centerline{\includegraphics[height=3cm]{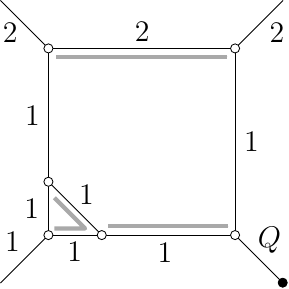}}
\end {myfigure}

So, we may assume all edges of $C$ have residue 1. It follows that all the edges incident with $v_2,v_3,v_4$ have residue 1 (Lemma~\ref{forbiddentriads}). 
If both of $v_2,v_3$ correspond to triads or both correspond to 111 triangles, then we apply Lemma~\ref{deleteadj}. 
Note that the assumptions of the lemma are satisfied: neither $w v_2$ nor $w v_3$ is an edge since $N(w) \cap V(C) = \{v_1\}$ and for the connectivity 
assumption we use Observation~\ref{obs4cuts} together with the definition of type~$\beta$. 
Thus we have the exceptional configuration in Lemma~\ref{deleteadj} so $w$ is connected by an edge to~$v_1$ and by a two-edge path 
to each of $v_2$, $v_3$. Consequently, the graph has a 3-edge-cut, or is isomorphic to~$V_8$, a contradiction. 
Similarly, we cannot have that $v_3$ and $v_4$ both correspond to 111 triangle or both correspond to triads (so the vertices $v_2,v_3,v_4$ of $V(G^{\Delta})$ alternate triangle, triad, triangle, or triad,
triangle, triad). It follows by Lemma~\ref{all1insq} that the subgraph of $G^{\bullet}$ corresponding to $C$ and its incident edges must be as on the right in Figure~\ref{reducesqares}. By the previous arguments,
there exists at most one ear on the periphery with length not congruent to 1 mod 3, namely $wv_1$. 
Now $v_1$ can correspond to a triad or to a triangle, so in Figure~\ref{reducesqares} in the right, we may assume 
that either $w=x$ or $w=y$. 
If $w=x$, we remove the ears in the order $a$, $b$, $c$, $d$. When removing ear labeled~$c$, we use Lemma~\ref{clever_choice} not to lose the 
bonus of the ear incident with~$w$ (this might possibly happen if it was an inequitable ear of lenght 0 mod 3). This way we lose bonus~48 (twelve ears of length 1 mod 3) and 
gain $8(1+2+1+2) = 48$. 
If $w=y$, the analysis is the same, only we need to remove ears in the order $c$, $d$, $a$, $b$, 
using Lemma~\ref{clever_choice} for removing~$a$. 
Altogether, the indicated partial ear decomposition gives a reducible subgraph.
%\bigskip

\paragraph{Case 3: There exists $Z \subseteq V(G^{\Delta})$ satisfying properties 1--5 in the statement of Lemma~\ref{alphabetagood}.}

%\smallskip

By Lemma~\ref{tame112}, no vertex of $Z$ corresponds to a 112 triangle.
First suppose that there exists an edge $yz$ of residue 2 with both ends in $Z$.  Note that by Lemma~\ref{outer2force}, both $y$ and $z$ must
  correspond to triads in $G^{\bullet}$ and each of $y,z$ will have one incident edge of residue 1 and two of residue~2.  By the fourth property in
  Lemma~\ref{alphabetagood}, we may assume (without loss) that all neighbours of~$y$ are in~$Z$.  In particular, this implies that there is a 2-edge
  path with vertex sequence $x,y,z$ with all vertices in $Z$ and both edges of residue~2.  Furthermore, $x$ will also be incident in $G^{\Delta}$ with
  exactly one edge of residue~1 and two of residue~2.  It follows from the first property in Lemma~\ref{alphabetagood} that at most one of $x,z$ is
  adjacent to $w$ in $G^{\Delta}$.  So, we may assume (without loss) that $x,w$ are not adjacent in $G^{\Delta}$.  Choose an edge $ux \in
  E(G^{\Delta})$ with $u \neq y$ so that $ux$ has residue~2. Since $x$ is not adjacent to $w$,  Lemma~\ref{tame112} implies that the vertex $u$ is
  also not associated with a 112 triangle in $G^{\bullet}$ so it must also be associated with a triad. 
  Note that by the fifth property in Lemma~\ref{alphabetagood}, $G^{\Delta} - \{yz\}$ is cyclically 4-edge-connected, thus $G^{\Delta} - \{yz, ux\}$ is cyclically 3-edge-connected. 
  Recall that each of $u$, $x$, $y$, and $z$ is incident with one edge of residue~1 and two edges of residue~2. 
  Next we consider the partial ear decomposition of $G^{\bullet}$ obtained by removing the ears associated with $yz$ and $ux$. 
  By this operation we loose five ears of length 2 mod 3 and four ears of length 1 mod 3. Total gain from this removal is at least $2 \cdot 8 \cdot 2$
  by Lemma~\ref{goodgain}. As $32 > 5\cdot 3 + 4\cdot 4$, we have indeed found a reducible subgraph. 
 
  So, we may now assume that every edge of $G^{\Delta}$ with both ends in $Z$ has residue~1.  Modify $Z$ to form the set $Z'$ by deleting from $Z$ any vertex which has $w$ as a neighbour and note that 
  $|Z \setminus Z'| \le 2$ (since $|d(Z)|=4$).  This implies that $d(Z') \le d(Z)+2 \le 6$. 
  If $Z'$ induces a subgraph with a vertex $x$ of degree at most $1$, then either $x$ has degree 2 in $Z$ and $x$ is adjacent to a vertex of $Z$ which is a  neighbour of $w$ (contradicting part 4 of Lemma~\ref{alphabetagood}), or $x$ is adjacent to two neighbours of $w$ (contradicting part 1 in Lemma~\ref{alphabetagood}).
  So $Z'$ still induces a subgraph of $G^{\Delta}$ with minimum degree 2, which means it induces a graph containing cycle.
   By Lemma~\ref{even2}, all edges incident with a vertex in $Z'$ have residue 1.
  If there are two adjacent vertices $v_1$ and $v_2$ in $Z'$ which are both associated with triangles or both associated with triads, then we have the configuration in
  Lemma~\ref{deleteadj} (we use Observation~\ref{obs4cuts} and Property~5 in Lemma~\ref{alphabetagood}). 
  Let $v_3$ and $v_4$ be the other two vertices in the 4-cycle $C$ containing $v_1$ and $v_2$ in this configuration. Now none 
  of the neighbours of $v_1$ and $v_2$ outside of $C$ is in $Z'$, so $v_3$ and $v_4$ must both be in $Z'$. It follows that $C$ is contained in $Z$, a
  contradiction to part 4 of Lemma~\ref{alphabetagood}.
  So, we may assume that the subgraph of~$G^{\Delta}$ induced by $Z'$ has a bipartition $(A,B)$ where the vertices in $A$ correspond to triangles, and
  the vertices in $B$ to triads.  
 
  If there is a vertex in $B$ which has all of its neighbours in $Z'$, then we have a contradiction to
  Lemma~\ref{forbid12}.
  Thus in the graph $G^{\Delta}[A \cup B]$, all vertices in~$B$ have degree~2, vertices in~$A$ have degree~$2$ or~$3$
  (and the total number of degree~2 vertices equals $d(Z')$). 
  Let $H$ be the graph obtained from~$G^{\Delta}[A \cup B]$ by suppressing all vertices in~$B$. (This is indeed a simple graph, as there are no 
  4-cycles in~$Z'$.) We have observed that $H$ is a simple graph with all degrees $2$ or~$3$. The number of edges of~$H$ plus the number of degree-2 vertices in~$H$
  is equal to $d(Z')$, which is at most $6$. A simple case analysis shows that there are only two such graphs: $H \cong K_4$ and $H \cong C_3$.

Therefore the only possibilities for $G^{\Delta}[Z']$ are a 6-cycle and $K_4$ with every edge subdivided exactly once. 
Both have the number of degree-2 vertices equal to~6, thus $d_{G^{\Delta}}(Z') = 6$ and so $|Z \setminus Z'| = 2$. 
It follows that $d(Z \cup \{w\}) = 3$ and $V(G^{\Delta}) \setminus Z$ consists of just two vertices.  
Therefore $V(G^{\Delta})$ may be expressed as the disjoint union of
$\{w\}$, three neighbours of $w$, and $Z'$.  To satisfy the assumption that $G^{\Delta}[Z]$ has girth $\ge 5$, there is 
a unique way how to add the vertices to each of the two possible variants of~$G^{\Delta}[Z']$. 
They lead to a graph that is isomorphic to either Petersen or Heawood graph.  Furthermore, every edge of $G^{\Delta}$ not incident with $w$ must have residue 1.
These last two graphs are resolved by the reducible subgraphs indicated by the partial ear decompositions in the figure below (the figure on the right
shows an embedding on the torus where opposite sides are identified as depicted). In both cases, the indicated partial ear decomposition has gain 72, there are
21 ears of $G$ with bonus 4 which are not ears of the resulting graph (for a loss of 84), but there are 3 newly formed ears with length 1 mod 3
which contribute 12.

\begin{figure}[htp]
\centerline{\includegraphics[width=13cm]{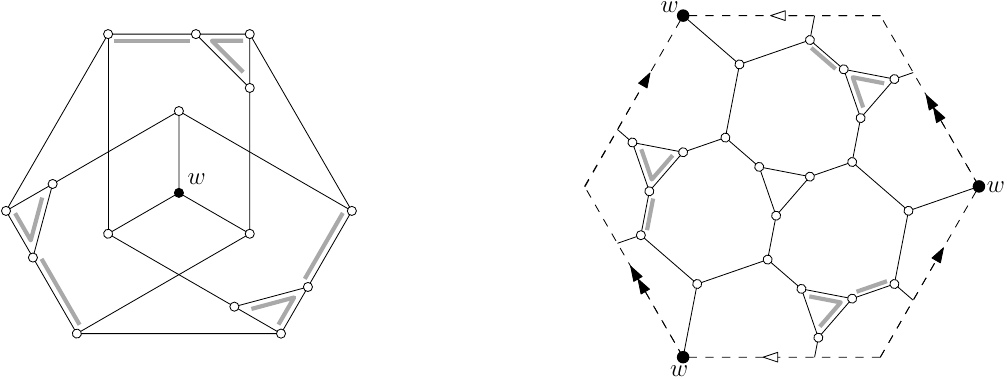}}
\end{figure}
\end{proof}

\section{Future Work}

While our methods require 3-edge-connected graphs, we believe the following conjecture is true.

\begin{conjecture}
Every 2-edge-connected graph $G$ has a 3-flow $\phi$ satisfying
\[ |\supp (\phi) | \ge \tfrac{5}{6} |E(G)|.\]
\end{conjecture}

It would be interesting to find ``approximative versions'' of the 5-flow, 4-flow and 3-flow conjectures:

\begin{problem}
  Suppose~$G$ is a graph with $m$ edges.
\begin{enumerate}
  \item How large can $\supp{\phi}$ be for a 5-flow~$\phi$, if $G$ is 2-edge-connected? (or 3-edge-connected?)
  \item How large can $\supp{\phi}$ be for a 4-flow~$\phi$, if $G$ is 2-edge-connected and has no Petersen minor?
  \item How large can $\supp{\phi}$ be for a 3-flow~$\phi$, if $G$ is 4-edge-connected? (or 5-edge-connected?)
\end{enumerate}
\end{problem}

In each of the above questions, Tutte~\cite{Tutte54, Tutte66} conjectures that the support may be of size $m$ (all of the edges get nonzero value).
We suggest weaker results that may be more approachable: for an appropriate function~$f$, prove that the support of some flow is at least $f(m)$.

Another interesting development would be to look for a group-connectivity version: every edge~$e$ has its ``forbidden'' value~$h(e)$
and we are looking for a flow~$\phi$ such that for as many edges as possible we have $\phi(e) \ne h(e)$.

\section*{Acknowledgement}
The authors would like to thank Dan Kr\'al', Bojan Mohar and Michael Tarsi for helpful ideas. Additionally, careful reading by an anonymous referee improved the exposition of this paper.

Matt DeVos was supported in part by an NSERC Discovery Grant (Canada).
Jessica McDonald was supported in part by a grant from the National Science Foundation, NSF-DMS-1600551.
Edita Rollov\'a was partially supported by project GA14-19503S of the Czech Science Foundation and by project LO1506 of the Czech Ministry of Education, Youth and Sports.
Robert \v{S}\'amal was partially supported by grant GA \v{C}R 19-21082S and by European Union’s Horizon 2020 research and innovation programme under the Marie Skłodowska-Curie grant agreement No 823748. 

\bibliographystyle{rs-amsplain}
\bibliography{flows-biblio}

\end{document}